\documentclass[abstract]{scrartcl}

\usepackage[utf8]{inputenc}
\usepackage[T1]{fontenc}
\usepackage{booktabs}
\usepackage{amsfonts}
\usepackage{amssymb}
\usepackage{amsthm}
\newtheorem{remark}{Remark}
\usepackage{soul}
\usepackage[normalem]{ulem}

\usepackage{graphicx}
\usepackage{subcaption}
\usepackage[english]{babel}
\usepackage{mathtools}
\mathtoolsset{showonlyrefs,showmanualtags}

\usepackage{pgfplots}
\pgfplotsset{compat=newest}
\usepgfplotslibrary{groupplots}
\usepgfplotslibrary{dateplot}

\usepackage{hyperref}
\hypersetup{
colorlinks,
linkcolor={blue!60!black},
citecolor={blue!60!black},
urlcolor={blue!60!black},
linktoc=page}

\usepackage[numbers]{natbib}
\usepackage{authblk} 

\usepackage{array}

\usepackage{todonotes}

\newcommand{\onelinewidth}{0.96\linewidth}
\newcommand{\halflinewidth}{0.48\linewidth}

\newcommand{\cC}{\mathcal{C}}

\newcommand{\cL}{\mathcal{L}}
\newcommand{\cO}{\mathcal{O}}
\newcommand{\cH}{\mathcal{H}}

\newcommand{\der}[2]{\frac{\partial{#1}}{\partial{#2}}}

\renewcommand{\d}{\mathrm{d}}

\newcommand{\R}{\mathbb{R}}

\newcommand{\jacobi}[1]{\{#1\}_{\eta}}
\newcommand{\jacobilr}[1]{\left\lbrace#1\right\rbrace_{\eta}}

\DeclareMathOperator{\ext}{ext}

\newtheorem{definition}{Definition}
\newtheorem{proposition}{Proposition}

\newtheorem{corollary}{Corollary}

\theoremstyle{definition}
\newtheorem{example}{Example}

\title{Numerical integration in celestial mechanics: a case for contact geometry}

\author[1,2]{Alessandro Bravetti}
\author[3]{Marcello Seri}
\author[4]{Mats Vermeeren}
\author[5]{Federico Zadra}
\affil[1]{Centro de Investigación en Matemáticas (CIMAT), Guanajuato, Mexico,}
\affil[2]{Instituto de Investigaciones en Matemáticas Aplicadas y en Sistemas (IIMAS--UNAM), Mexico City, Mexico,
 \authorcr\normalsize \texttt{alessandro.bravetti@cimat.mx}}
\affil[3]{Bernoulli Institute for Mathematics, Computer Science and Artificial Intelligence, \authorcr\small Groningen, The Netherlands,
\authorcr\normalsize \texttt{m.seri@rug.nl}}
\affil[4]{Technische Universität Berlin, Germany, \authorcr\normalsize \texttt{vermeeren@math.tu-berlin.de}}
\affil[5]{Bernoulli Institute for Mathematics, Computer Science and Artificial Intelligence, \authorcr\small Groningen, The Netherlands,
\authorcr\normalsize \texttt{f.zadra@rug.nl}}

\date{}

\begin{document}

\maketitle
\abstract{
    \noindent
    Several dynamical systems of interest in celestial mechanics can be written in the form of {a Newton equation with time--dependent damping, linear in the velocities}.
    For instance, the modified Kepler problem, the spin--orbit model and the Lane--Emden equation all belong to such class.
    In this work we start an investigation of these models from the point of view of contact geometry. In particular we focus on the (contact) Hamiltonisation of these models and on the construction of the corresponding geometric integrators.

    \medskip
    \noindent
    \textbf{Keywords:} contact geometry, geometric integrators, modified Kepler, spin--orbit, Lane--Emden

    \noindent
    \textbf{MSC2010:} 65D30, 34K28, 34A26
}

\section{Introduction}

Geometric methods in the study of dynamical systems have proven to be useful both for analytical and for numerical investigations~\cite{hairer2006geometric,marsden2001discrete,sanz1992symplectic}.
When the system has a Hamiltonian form, it is possible to exploit the geometric structure of the problem and provide powerful tools to study and classify the dynamics.
One of the central properties of Hamiltonian systems is the conservation of energy. Therefore, it is not surprising that finding a Hamiltonian formulation for second order differential equations with non--conservative forces {is considered a very hard, if not hopeless, problem}.

In this work we consider the important class of systems
\begin{equation}%
	\label{eq:geneq}
    \ddot q + \frac{\partial V(q,t)}{\partial q}+f(t)\dot q=0\,.
\end{equation}
This class arises from the Newtonian mechanics of systems with time--varying non--conservative forces,
and includes systems of primary interest in celestial mechanics such as the modified Kepler problem, the spin--orbit model and the Lane--Emden equation, which will be analysed in {more} detail below.

Our first crucial remark is that all systems of the form~\eqref{eq:geneq} can be given a Hamiltonian formulation in the context of contact geometry: Section~\ref{sec:contact} will be devoted to review contact Hamiltonian systems and make the previous statement precise.

We also argue that this geometric reformulation of equation~\eqref{eq:geneq} can be useful in order to study the dynamics, both from an analytical and from a numerical perspective.
Referring the analytical study to future investigations, we start in this paper a thorough investigation of numerical methods for the analysis of such systems based on contact geometry. In particular, in Section~\ref{sec:integrator} we develop both Lagrangian and Hamiltonian integrators for contact systems, and then
in Section~\ref{sec:examples} we consider
the modified Kepler problem, the spin--orbit model and Lane--Emden equation, and show with numerical tests that the contact perspective can improve over previously presented methods to study the dynamics numerically.

Finally, in Section~\ref{sec:conclusions} we provide a summary of results and discuss future directions.

\section{Contact dynamics}%
	\label{sec:contact}
In this section we review briefly the basic notions of contact geometry and dynamics, both from the Hamiltonian and from the Lagrangian perspective.

\subsection{Contact geometry}

Contact geometry is the ``odd--dimensional cousin'' of symplectic geometry. However, it has received much less attention in the scientific literature until very recently (with the exception of contact topology~\cite{geiges2008introduction}).
One of the reasons may be that
the definition of a contact manifold is somewhat involved, and therefore we decide here to present the general definition together with a more restrictive one, which is all is needed for the present work.
{In general,} a \emph{contact manifold} is a $(2n+1)$--dimensional manifold $M$ endowed with a contact structure $\mathcal{D}$, that is, a maximally non--integrable distribution of hyperplanes, meaning that its integrable submanifolds have dimension at most $n$.
Here, borrowing the terminology from~\cite{ciaglia2018contact}, we restrict to the following important case.
\begin{definition}%
	\label{contactM}
    An \emph{exact contact manifold} is a $(2n+1)$--dimensional manifold $M$ endowed with a contact structure $\mathcal{D}$ which is given globally as the kernel of a $1$--form $\eta$ satisfying the non--degeneracy condition $\eta\wedge(\d\eta)^n\neq 0$.
\end{definition}
It is easy to verify that any exact contact manifold is a contact manifold, but that the converse is not true (for instance, if the manifold $M$ is not orientable, then it can admit a contact structure, but it cannot admit a global $1$--form satisfying the non--degeneracy condition needed in the definition of an exact contact manifold).
In this work we will always assume that $(M,\mathcal{D})$ is an exact contact manifold.

Notice also that since $\mathcal{D}=\ker(\eta)$, we have in fact an equivalence class of $1$--forms that define the same $\mathcal{D}$. These are given by re--scaling $\eta$ by multiplication by a non--vanishing function.
In the following definition we make this statement precise.

\begin{definition}%
	\label{contactomorphism}
    Given an exact contact manifold $(M,\mathcal{D})$ and a representative $\eta$ such that $\mathcal{D}=\ker(\eta)$, a diffeomorphism $f:M\rightarrow M$ is called a \emph{contactomorphism} if it satisfies $f^*\eta=\rho_f\,\eta$, where $f^*$ is the pullback induced by $f$ and $\rho_f:M\rightarrow\mathbb{R}$ is any non--vanishing function.
    Analogously, a vector field $X\in\mathfrak{X}(M)$ is called a \emph{contact vector field} if $\pounds_X\eta=\rho_X\eta$, where $\pounds$ is the Lie derivative and $\rho_X$ is any function (the case when $\rho_X=0$ is called sometimes a \emph{strict contact vector field}).
\end{definition}

\begin{example}
    The standard example of a contact manifold is $M=\mathbb{R}^{2n+1}$ with
    Cartesian coordinates $(q^i,p_i,s)$, for $i=1,\dots,n$ and
    the standard contact structure $\mathcal{D}=\ker\left(\d s-\sum_{a=1}^n p_a\d q^a\right)$.
\end{example}

\subsection{Contact Hamiltonian systems}

We are now in the position to define Hamiltonian systems on (exact) contact manifolds.
To do so, we will first fix a representative $1$--form and then define the given Hamiltonian vector field associated to any function $\cH:M\times \mathbb{R}\rightarrow \mathbb{R}$, where the additional dimension is used to account for time--dependence.

\begin{definition}
    Let $(M,\mathcal{D})$ be an exact contact manifold and fix a representative $1$--form $\eta$ generating the contact structure. For any function
    $\cH:  M \times \R \rightarrow \mathbb{R}$
    we define the corresponding \emph{contact Hamiltonian vector field} $X_{\cH}$ by
    \begin{equation}%
	\label{XCH}
        \pounds_{X_{\cH}}\eta=\rho_H\eta,\qquad \iota_{X_{\cH}} \eta=-\cH\,,
    \end{equation}
    where $\pounds$ is the Lie derivative, and $\iota$ represents the interior product.
\end{definition}

Note that the first condition in~\eqref{XCH} is the requirement that $X_{\cH}$ be a contact vector field, while the second condition is the association between the vector field and its Hamiltonian function.

We refer the interested reader to {the classical textbook~\cite{arnold} for an introduction to contact geometry, and to~\cite{Bravetti2017a,bravetti2018contact} for an overview of physical applications. More specifically, in~\cite{bravetti2017contact,deleon2017cos,de2019singular,gaset2019contact,gaset2019new,valcazar2018contact}
one can find a detailed study of the contact geometry of dissipative systems in classical mechanics and field theories, while in~\cite{ciaglia2018contact} 
an approach to study dissipative systems in quantum mechanics via contact systems is investigated.}
Here we limit ourselves to collect some relevant properties in the following proposition.

\begin{proposition}
    In the neighbourhood of any point $P\in M$, there always exist local coordinates $(q,p,s)$, called \emph{Darboux coordinates}, such that $\eta=ds-p_a\d q^a$. In such coordinates $X_{\cH}$ takes the form
    \begin{equation}%
	\label{CXH}
        X_{\cH}=
        \frac{\partial \cH}{\partial p_a}\frac{\partial }{\partial q^a}
        +\left(-\frac{\partial \cH}{\partial q^a}-p_a\frac{\partial \cH}{\partial s}\right)\frac{\partial }{\partial p_a}
        +\left(\frac{\partial \cH}{\partial p_a}p_a-\cH\right)\frac{\partial }{\partial s} ,
    \end{equation}
    where summation over repeated indices is assumed from now on. 
    The corresponding contact Hamiltonian equations are given by
    \begin{align}
        \dot q^a &=
        \frac{\partial \cH}{\partial p_a}%
	\label{CHeq1}\\
        \dot p_a &=-\frac{\partial \cH}{\partial q^a}-p_a\frac{\partial \cH}{\partial s}%
	\label{CHeq2}\\
        \dot s &=\frac{\partial \cH}{\partial p_a}p_a-\cH\label{CHeq3}\,.
    \end{align}
\end{proposition}
\noeqref{CHeq1}\noeqref{CHeq2}\noeqref{CHeq3}
{The proof of the first part of the above proposition (Darboux theorem) can be found e.g.~in~\cite{geiges2008introduction}. The second part (the coordinate expression of $X_{\cH}$) can be  obtained by expressing the two conditions in~\eqref{XCH} in Darboux coordinates and using Cartan's identity.}

The following property guarantees that contact vector fields come with an associated Hamiltonian function, and will be essential for the construction of the contact Hamiltonian integrators {(for the proof see e.g.~\cite{valcazar2018contact})}.

\begin{proposition}[Isomorphism between functions and contact vector fields]\label{prop:isomorphism}
    Let $(M,\mathcal{D})$ be an exact contact manifold and fix a representative $1$--form $\eta$. Then $\eta$ induces an isomorphism of Lie algebras, between the Lie algebra of contact vector fields (infinitesimal generators of contactomorphisms) with the standard Lie bracket and the Lie algebra of functions on $M$ with the \emph{Jacobi bracket} defined as
    \begin{equation}%
	\label{Jacobibracket}
        \jacobi{g,f} = \iota_{[X_g,X_f]} \eta\,.
    \end{equation}
    In particular, to every contact vector field $X$ we can associate uniquely a contact Hamiltonian $\cH_X$ by means of $\iota_{X} \eta=-\cH_X$ (cf.~the second condition in~\eqref{XCH}).
In Darboux coordinates, the Jacobi bracket reads
\begin{equation}
    \jacobi{g,f} = \left(g \frac{\partial f}{\partial s} - \frac{\partial g}{\partial s} f\right) + p_{\mu} \left(\frac{\partial g}{\partial s} \frac{\partial f}{\partial p_{\mu}} - \frac{\partial g}{\partial p_{\mu}} \frac{\partial f}{\partial s}\right) + \left(\frac{\partial g}{\partial q^{\mu}} \frac{\partial f }{\partial p_{\mu}} - \frac{\partial g}{\partial p_{\mu}} \frac{\partial f }{\partial q^{\mu}}\right). \label{Jacobibracketcoordinates}
\end{equation}
\end{proposition}

Even though the Jacobi bracket can be understood as a generalisation of the Poisson bracket, the equations of motion are not the same as for the Poisson bracket.
As one can verify by direct calculation, there holds $\frac{df}{dt} = \jacobi{f,H} - f \frac{\partial H}{\partial s}$ for any function $f(q,p,s)$.

The following proposition is the starting point for our analysis of equation~\eqref{eq:geneq} in terms of contact geometry.

\begin{proposition}
\label{prop:mainCH}
    Equation~\eqref{eq:geneq} corresponds to
    the flow of the contact Hamiltonian
    \begin{equation}%
	\label{eq:mainCH}
        \cH(p,q,s,t)=\sum_{a=1}^n\frac{p_a^2}{2}+V(q,t)+f(t)\,s\,.
    \end{equation}
\end{proposition}
\begin{proof}
    To prove the above statement, observe that the Hamiltonian equations~\eqref{CHeq1}--\eqref{CHeq3} in this case read
    \begin{eqnarray}
        \dot q_a&=&
        p_a\label{mainCHeq1}\\
        \dot p_a&=&-\frac{\partial V(q,t)}{\partial q^a}-f(t)p_a\label{mainCHeq2}\\
        \dot s&=&\sum_{a=1}^n\frac{p_a^2}{2}-V(q,t)-f(t)\,s\label{mainCHeq3}\,.
    \end{eqnarray}
    \noeqref{mainCHeq1}\noeqref{mainCHeq2}\noeqref{mainCHeq3}
    One can immediately check that the system~\eqref{mainCHeq1}--\eqref{mainCHeq2}, gives exactly the
    system~\eqref{eq:geneq}, while equation~\eqref{mainCHeq3} decouples from the rest.
\end{proof}

Using contact geometry, we have immediately obtained a ``Hamiltonisation'' of all the dynamical systems of the form~\eqref{eq:geneq}.
This fact should not be underestimated: a Hamiltonian structure for all such systems
allows us to benefit from the theory of Hamiltonian systems (extended to the contact case) and its powerful analytical and numerical tools.
For instance, one can apply weak--KAM theorems and variational methods, as done e.g.~in~\cite{cannarsa2019herglotz,liu2018contact,wang2016implicit,wang2019aubry}.

It is important to also compare the simplicity and generality of the formulation provided here against previous attempts in the literature.
For instance, in~\cite{gkolias2017hamiltonian} an algorithm for the symplectic Hamiltonisation of systems of the type~\eqref{eq:geneq} has been provided. However, the construction suggested there is based on a non--trivial reparameterisation that requires solving an additional differential equation in order to obtain the new time variable (which in many cases cannot be done exactly, cf.~\cite{gkolias2017hamiltonian}). 
We stress that in our analysis we do not encounter {any} such  complication. 

\subsection{Herglotz' variational principle}%
	\label{sec:varPri}

As for symplectic Hamiltonian systems, the dynamics of contact Hamiltonian systems can be characterised by a variational principle. This was originally published by Herglotz in a set of lecture notes~\cite{herglotz1930vorlesungen}, which might explain why it has received relatively little attention. A modern discussion of Herglotz' variational principle can be found for example in~\cite{georgieva2003generalized,vermeeren2019contact} (see also~\cite{gaset2019contact,lazo2017action,lazo2018action} for extensions to field theories).

\begin{definition}[Herglotz' variational principle]
	Let $Q$ be an $n$--dimensional manifold with local coordinates $q^i$ and let $\cL: \R \times TQ \times \R \rightarrow \R$. For any given smooth curve $q:[0,T] \rightarrow Q$ we consider the initial value problem
	\begin{equation}%
	\label{eq:action}
	\dot{s} = \cL(t,q(t),\dot{q}(t),s), \qquad s(0) = s_\mathrm{init}.
	\end{equation}
	Then the value $s(T)$ is a functional of the curve $q$. We say that $q$ is \emph{a critical curve} if $s(T)$ is invariant under infinitesimal variations of $q$ that vanish at the boundary of $[0,T]$.
\end{definition}

If the Lagrange function does not depend on $s$, then we find
\[ s(T) = \int_0^T  \cL(t,q(t),\dot{q}(t)) \, \d t , \]
which is the usual action functional from symplectic mechanics. Hence the classical formulation of Lagrangian mechanics is a special case of Herglotz' variational principle.

\begin{proposition}
	Critical curves for the Herglotz' variational principle are characterised by the following generalised Euler--Lagrange equations:
	\begin{equation}%
	\label{eq:gEL}
	\der{\cL}{q^a} - \frac{\d}{\d t} \der{\cL}{\dot{q}^a} + \der{\cL}{s}\der{\cL}{\dot{q}^a} = 0.
	\end{equation}
\end{proposition}

\begin{proof}
The generalised Euler--Lagrange equations can be derived by solving the differential equation
\[ \delta \dot{s} = \der{\cL}{q^a} \delta q^a + \der{\cL}{\dot{q}^a} \delta \dot{q}^a + \der{\cL}{s} \delta s . \]
We find
\begin{equation}%
	\label{eq:variation}
\begin{split}
\delta s(T) &= \exp\left(\int_0^T \der{\cL}{s} \right) \int_0^T \left(\der{\cL}{q^a} - \frac{\d}{\d t} \der{\cL}{\dot{q}^a} + \der{\cL}{s}\der{\cL}{\dot{q}^a} \right) \delta q^a \exp\left(-\int_0^T \der{\cL}{s} \right) \d t \\
&\quad + \der{\cL}{\dot{q}^a}(T) \delta q^a(T) + \left(\delta s(0) - \der{\cL}{\dot{q}^a}(0) \delta q^a(0) \right) \exp\left(\int_0^T \der{\cL}{s} \right),
\end{split}
\end{equation}
where the boundary terms on the second line vanish because variations leave the endpoints fixed.
\end{proof}

If we restrict our attention to solutions to the generalised Euler--Lagrange equations, but allow variations of the endpoint, equation~\eqref{eq:variation} reduces to
\[ \delta s(T) - p_a(T) \delta q^a(T) = \left(\delta s(0) - p_a(0) \delta q^a(0) \right) \exp\left(\int_0^T \der{\cL}{s} \right), \]
where $p_a = \der{\cL}{\dot{q}^a}$. This means that the flow consists of contact transformations with respect to the 1--form $\eta = \d s - p_a \d q^a$.

One can verify that for the Hamiltonian $\cH(t,q,p,s) = p \dot{q} - \cL(t,q,\dot{q},s)$, where $\dot{q}$ is written in function of $p$ and $q$, the equations~\eqref{CHeq1}--\eqref{CHeq3} are equivalent to the system consisting of equation~\eqref{eq:action} and the generalised Euler--Lagrange equations~\eqref{eq:gEL}.

There is a natural discretisation of Herglotz' variational principle, which was introduced in~\cite{vermeeren2019contact}.

\begin{definition}[Discrete Herglotz' variational principle]
	Let $Q$ be an $n$--dimensional manifold with local coordinates $q^i$ and let $L: \R \times Q^2 \times \R^2 \rightarrow \R$. For any given discrete curve $q: \{0,\ldots,N\} \rightarrow Q$ we consider the initial value problem
	\begin{equation}%
	\label{eq:daction}
	s_{k+1} = s_k + \tau L(k\tau,q_k,q_{k+1},s_k, s_{k+1}), \qquad s_0 = s_\mathrm{init}.
	\end{equation}
	Then the value $s_N$ is a functional of the discrete curve $q$. We say that $q$ is 
	\emph{a critical curve} if
	\[ \der{s_N}{q_k} = 0 \qquad \forall k \in \{1,\ldots,N-1\}\,. \]
\end{definition}
Equivalently, we can require that $\der{s_{k+1}}{q_k} = 0$ for all $k \in \{1,\ldots,N-1\}$. By elementary calculations, this gives us the discrete generalised Euler--Lagrange equations. As in the conventional discrete calculus of variations, they can be formulated as the equality of two formulas for the momentum~\cite{marsden2001discrete}.

\begin{proposition}%
	\label{prop-dgEL}
Let
\begin{align*}
	p_k^- &= \frac{ \displaystyle \der{}{q_k} L((k-1)\tau,q_{k-1},q_k,s_{k-1}, s_k) }{ \displaystyle 1 - \tau \der{}{s_k} L((k-1)\tau,q_{k-1},q_k,s_{k-1}, s_k) }\,, \\
	p_k^+ &= -\frac{ \displaystyle \der{}{q_k} L(k\tau,q_k,q_{k+1},s_k, s_{k+1}) }{ \displaystyle 1 + \tau \der{}{s_k} L(k\tau,q_k,q_{k+1},s_k, s_{k+1}) }\,.
	\end{align*}
Then solutions to the discrete Herglotz variational principle are characterised by
\[ p_k^- = p_k^+ . \]
Furthermore, the map $(q_k,p_k,s_k) \mapsto (q_{k+1},p_{k+1},s_{k+1})$ induced by a critical discrete curve preserves the contact structure $\ker(\d s - p_a \d q^a)$.
\end{proposition}
{For a proof of the above proposition we refer to~\cite{vermeeren2019contact}.}
Without loss of generality it is possible to take the discrete Lagrange function depending on only one instance of $s$: $L(k\tau,q_k,q_{k+1},s_k)$. Then the discrete generalised Euler--Lagrange equations read
\begin{align*}
    &\der{}{q_k} L(k\tau,q_k,q_{k+1},s_k) \\
    &+ \der{}{q_k} L((k-1)\tau,q_{k-1},q_k,s_{k-1}) \left( 1 + \tau \der{}{s_k} L(k\tau,q_k,q_{k+1},s_k) \right) = 0\,.
\end{align*}

\section{Integrators}%
	\label{sec:integrator}
One of the advantages of having a contact Hamiltonian structure is that  many ideas for geometric integrators for symplectic systems can be carried over with relatively small effort. A study of variational integrators in the contact case has been started recently in~\cite{vermeeren2019contact}.
Here we develop higher order variational and Hamiltonian integrators for contact systems.

\subsection{Lagrangian}%
	\label{sec:Lagrangian}

Like in the symplectic case, we can construct higher order contact variational integrators using a Galerkin discretisation. In such a discretisation, the set of curves on the time interval of one step $\cC([0,\tau],Q) = \{ q: [0,\tau] \rightarrow Q \mid q(0) = q_0, q(\tau) = q_1 \}$ is replaced by a finite--dimensional space of polynomials
\[ \cC^\ell([0,\tau],Q) = \left\{ q \in \cC([0,\tau],Q) \mid q \text{ a polynomial of degree at most } \ell \right\} . \]
To paremetrise this space we introduce $\ell+1$ control points $d_0 < d_1 < \ldots < d_\ell$, where $d_0 = 0$ and $d_\ell = 1$. If for each of these control points a value $q(\tau d_i) = q_i$ is prescribed, then the polynomial $q \in \cC^\ell([0,\tau],Q)$ is uniquely determined. We denote by $\hat{q}(\cdot ; q_0,\ldots,q_\ell,\tau)$ the polynomial thus obtained. 

Given a continuous Lagrangian $\cL$, we would like to define $s: [0,\tau] \rightarrow \R$ by specifying an initial condition $s(0) = s_0$ and setting
\[ \dot{s}(t) = \cL\!\left( t, \hat{q}(t ; q_0,\ldots,q_\ell,\tau), \dot{\hat{q}}(t ; q_0,\ldots,q_\ell,\tau), s(t) \right). \]
To approximate $s$, and in particular $s(\tau)$, we use an explicit Runge--Kutta method of order $u$ with coefficients $a_{ij}$, $b_i$ and $c_i$, i.e.\@ we calculate
\[ k_i = \sum_i \tau \cL \!\left(t_0 + \tau c_i, \hat{q}(\tau c_i ; q_0,\ldots,q_\ell,\tau), \dot{\hat{q}}(\tau c_i ; q_0,\ldots,q_\ell,\tau), s_0 + \sum_j a_{ij} k_j \right) \]
and set
\[ s_1(s_0,q_0,\ldots,q_\ell,\tau) = \sum_i b_i k_i(s_0,q_0,\ldots,q_\ell,\tau).\]
The discrete Lagrangian is defined by finding a critical value of $s_1$ and subtracting $s_0$ to match the formulation of the Herglotz variational principle:
\[L(q_0,q_\ell,s_0,\tau) = \ext_{q_1,\ldots,q_{\ell-1}} \!\left(\frac{ s_1(s_0,q_0,\ldots,q_\ell,\tau)-s_0}{\tau} \right). \]
where $\ext_{q_1,\ldots,q_{\ell-1}} $ denotes the critical value with respect to variations of $q_1,\ldots,q_{\ell-1}$.

\begin{remark}
Based on numerical evidence from the symplectic version of this construction~\cite{ober2015construction}, we expect the variational integrator defined by this Lagrangian to be of order $\min(2\ell,u)$, where $\ell$ is the degree of the polynomials and $u$ the order of the Runge--Kutta method. Hence the order of the integrator can be twice the degree of the polynomial approximation. A general proof of this fact is the topic of a future work.
\end{remark}

\begin{example}
We use second order polynomials $q\in \cC^2([0,1],Q)$ and control points $d_0 = 0$, $d_1 = \frac{1}{2}$ and $d_2 = 1$. If $q(\frac{\tau i}{2}) = q_i$ for $i \in \{0,1,2\}$, then $\hat q$ is given by
\[\hat q(\tau t) = 2 q_0 (t - \tfrac{1}{2}) (t - 1) - 4 q_1 t (t-1) + 2 q_2 t (t - \tfrac{1}{2}) \]
and its derivative by
\[ \dot{\hat q}(\tau t) = \frac{2 q_0 + 2 q_2}{\tau} (t - \tfrac{1}{2}) + \frac{2 q_0 - 4 q_1}{\tau} (t-1) + \frac{2 q_2 - 4 q_1}{\tau} t .\]
In particular we have
\begin{align*}
\dot{\hat q}(0) &= \frac{-3 q_0 + 4 q_1 - q_2}{\tau} \\
\dot{\hat q}(\tfrac{\tau}{2}) &= \frac{q_2 - q_0}{\tau} \\
\dot{\hat q}(\tau) &= \frac{q_0 - 4 q_1 + 3 q_2}{\tau}
\end{align*}

We use the classical fourth order Runge--Kutta method to calculate $s_1$, approximating $s(\tau)$ as a function of $(q_0,q_1,q_2,s_0)$:
{
\setlength{\tabcolsep}{0pt}
\renewcommand{\arraystretch}{2.2}
\[\begin{array}{r @{} >{\displaystyle}c<{,} >{\displaystyle}c<{,} >{\displaystyle}c<{,} >{\displaystyle}c @{} l}
k_1 = \tau \cL\bigg(& t_0 & q_0 & \frac{-3 q_0 + 4 q_1 - q_2}{\tau} & s_0 &\bigg) \\
k_2 = \tau \cL\bigg(& t_0 + \frac{1}{2}\tau & q_1 & \frac{q_2 - q_0}{\tau} & s_0 + \frac{1}{2} k_1 &\bigg) \\
k_3 = \tau \cL\bigg(& t_0 + \frac{1}{2}\tau & q_1 & \frac{q_2 - q_0}{\tau} & s_0 + \frac{1}{2} k_2 &\bigg) \\
k_4 = \tau \cL\bigg(& t_0 + \tau & q_2 & \frac{q_0 - 4 q_1 + 3 q_2}{\tau} & s_0 + k_3 &\bigg)
\end{array}\]
}
and
\[ s_1 = s_0 + \frac{k_1 + 2k_2 + 2k_3 + k_4}{6} . \]
This gives us the discrete Lagrangian
\[ L(q_0,q_2,s_0,\tau) = \ext_{q_1} \left( \frac{s_1(q_0,q_1,q_2,s_0,\tau) - s_0}{\tau} \right)  = \ext_{q_1} \left( \frac{k_1 + 2k_2 + 2k_3 + k_4}{6 \tau} \right), \]
from which a difference equation for $q$ is obtained by the discrete Herglotz variational principle (see Proposition~\ref{prop-dgEL}).
\end{example}

\subsection{Hamiltonian}%
	\label{sec:Hamiltonian}
Here we review the standard splitting procedure for the generation of higher order methods for separable flows and then apply it to the case of contact Hamiltonian systems in which the Hamiltonian can be split into the sum of different pieces that can be integrated exactly (see also~\cite{mclachlan2002splitting}).

\subsubsection{Splitting methods}
We begin with a brief review of the splitting method, following closely~\cite{yoshida1990construction}, {to which we refer for the proofs of Propositions~\ref{2nd}, \ref{prop:exact} and \ref{prop:approx}}. First of all, we have the following definition and a related important property.

\begin{definition}
    We say that a vector field $X$ is \emph{exactly integrable} if there exists a solution to the differential equation $\dot x=X(x),\, x(0)=x_0$, given by $x(t)=\exp(t X)x_0$, that can be explicitly written in closed form.
\end{definition}

\begin{proposition}[2nd--order integrator]\label{2nd}
    If a vector field $X(x)$ can be split as a sum
    \begin{equation}
        X(x) = \sum_{i=1}^n Y_i(x), \label{sumofvectorfield}
    \end{equation}
    where each of the vector fields $Y_i(x)$ is exactly integrable,
    then
    \begin{equation}
        S_2(\tau)=
        e^{\frac{\tau}{2} Y_1} e^{\frac{\tau}{2} Y_2} \cdots e^{{\tau} Y_n} \cdots e^{\frac{\tau}{2} Y_2} e^{\frac{\tau}{2} Y_1}
        \label{secondorderintegrator},
    \end{equation}
    is a second order integrator for the differential equation
    \begin{equation}
        \dot{x} = X(x). \label{diffeq}
    \end{equation}
\end{proposition}

\begin{remark}
    Proposition~\ref{2nd} holds also for non--integrable vector fields $Y_i$.
    However, the requirement of exact integrability is crucial to be able to implement the corresponding integrators.
\end{remark}

Based on a repeated use of the second order integrator~\eqref{secondorderintegrator} with appropriately changed step sizes, Yoshida~\cite{yoshida1990construction} developed two different algorithms to construct integrators of any even order;
the difference between the two is that
the first one uses exact coefficients to calculate the {rescaled} time steps, while the second one uses approximated coefficients. Although the first method in principle is more accurate, the latter is sometimes preferred because
it involves fewer calculations.

We can summarise these two approaches in the following statements.
\begin{proposition}[Integrator with exact coefficients]\label{prop:exact}
    If $S_{2n}(\tau)$ is an integrator of order $2n$, then the map
    \begin{equation}
        S_{2n+2} (\tau) = S_{2n}(z_1\tau) S_{2n}(z_0\tau) S_{2n}(z_1\tau), \label{2n+2orderexact}
    \end{equation}
    with $z_0$ and $z_1$ given by
    \begin{equation}%
	\label{z0z1}
        z_0(n)=-\frac{2^{\frac{1}{2n+1}}}{2-2^{\frac{1}{2n+1}}}, \hspace{0.5cm} z_1(n)=\frac{1}{2-2^{\frac{1}{2n+1}}};
    \end{equation}
    is an integrator of order $2n+2$.
\end{proposition}
\begin{proposition}[Integrator with approximated coefficients]\label{prop:approx}
    There exist $m\in\mathbb{N}$ and a set of real coefficients $\{w_j\}^m_{j=0}$ such that the map
    \begin{equation}
        S^{(m)}(\tau)=S_2(w_m\tau) S_2(w_{m-1}\tau) \cdots S_2(w_{0}\tau) \cdots S_2(w_{m-1}\tau) S_2(w_m\tau), \label{msymm}
    \end{equation}
    is an integrator of order $2n$.
\end{proposition}

The proof of Proposition~\ref{prop:approx} is constructive~\cite{yoshida1990construction}, and the coefficients are obtained as approximated solutions to an algebraic equation derived from the Baker--Campbell--Hausdorff formula
(see also equation~\eqref{eq:BCH} below).
Table~\ref{numsol} lists values of the coefficients $\{w_j\}^m_{j=0}$ for three $6^{th}$--order integrators, labeled A, B and C. Note that $w_0 := 1 - 2\sum_{j=1}^m w_i$.
\begin{table}[h!]
    \centering
    \begin{tabular}{cccc}
        \toprule
        & A & B & C \\
        \midrule
        $w_0$ & $1.315186320683906$& $2.37635274430774$
        & $2.3894477832436816$
        \\
        \midrule
        $w_1$ & $-1.17767998417887$ & $-2.13228522200144$ & $0.00152886228424922$ \\
        $w_2$ & $0.235573213359357$ & $0.00426068187079180$ & $-2.14403531630539$  \\
        $w_3$ & $0.784513610477560$ & $1.43984816797678$ & $1.44778256239930$ \\
        \bottomrule
    \end{tabular}
        \caption{The coefficients $w_{i}$ for three $6^{th}$--order integrators.}%
	\label{numsol}
\end{table}

\subsubsection{Contact integrators of order \texorpdfstring{$2n$}{2n}}
Let us now apply the above splitting schemes to derive contact integrators, i.e.~integrators that preserve the contact structure.
The time evolution in this case is given by a contact Hamiltonian vector field, and thus the flow is a contact map, as explained in Section~\ref{sec:contact}.
As a direct consequence of the splitting method, we have the following result.
\begin{proposition}%
	\label{prop:2ndcontact}
    Let the contact Hamiltonian be separable into the sum of functions
    \begin{equation}
        \mathcal{H}(q^i,p_i,s) = \sum_{j=1}^n \phi_j(q^i,p_i,s), \label{hamsum}
    \end{equation}
    such that each of the
    vector fields $X_{\phi_j}$
    is exactly integrable.
    Then the integrator
    \begin{equation}
        S_{2}(\tau)=e^{\frac{\tau}{2} X_{\phi_1}} e^{\frac{\tau}{2} X_{\phi_2}} \cdots e^{{\tau} X_{\phi_n}} \cdots e^{\frac{\tau}{2} X_{\phi_2}} e^{\frac{\tau}{2} X_{\phi_1}} \label{secondinteghamil}
    \end{equation}
    is a second order contact integrator.
\end{proposition}
\begin{proof}
    The fact that~\eqref{secondinteghamil} is a second order integrator follows directly from Proposition~\ref{2nd}.
    Moreover, each map $e^{\frac{\tau}{2} X_{\phi_j}}$
    is a contact map because by definition it is the flow of a contact Hamiltonian vector field. Being the composition of contact maps, $S_2(\tau)$ is a contact transformation itself.
\end{proof}

\begin{corollary}
We can construct contact integrators of any even order.
\end{corollary}
\begin{proof}
Such a construction is obtained by combining Proposition~\ref{prop:2ndcontact} with either Proposition~\ref{prop:exact} or Proposition~\ref{prop:approx}.
\end{proof}

\begin{remark}
The general question of finding all contact Hamiltonian systems admitting a splitting into exactly integrable pieces has been addressed in~\cite{mclachlan2002splitting}.
\end{remark}

\subsubsection{Modified Hamiltonian and error estimation}%
	\label{sec:modifiedhamiltonian}
One of the most powerful techniques to study the long--time behaviour of symplectic or contact integrators is the so--called \emph{backward error analysis}, that is, the study of modified differential equations that are exactly traced by the discrete maps of the integrator.

Propositions~\ref{prop:isomorphism} and~\ref{prop:2ndcontact} suggest that any contact integrator has an associated \emph{modified Hamiltonian}, meaning that the numeric integration follows exactly the flow of a different contact Hamiltonian.
Below we show the construction for a second order integrator for a time--dependent Hamiltonian of the type~\eqref{eq:mainCH}. 

We consider $\mathcal{H} = A + B + C$, then according to Proposition~\ref{prop:2ndcontact}, we have the second order integrator
\begin{align}
S_{2} (\tau) = &\exp\left(\frac{\tau}{2} \frac{\partial}{\partial t} \right) \exp\left(\frac{\tau}{2} X_C \right)\exp\left(\frac{\tau}{2} X_B \right) \notag \\
& \times \exp\left({\tau} X_A \right) \exp\left(\frac{\tau}{2} X_B \right)\exp\left(\frac{\tau}{2} X_C \right)\exp\left(\frac{\tau}{2} \frac{\partial}{\partial t} \right)\,, \label{secondintegrator}
\end{align}
where we stress that the {first and last} terms,  $\exp\left(\frac{\tau}{2} \frac{\partial}{\partial t} \right)$,  are needed only in the case of non--autonomous Hamiltonians.

The Baker--Campbell--Hausdorff (BCH) formula provides a closed expression to compute the product of exponentials of any two non--commutative operators $X$ and $Y$ in the Lie algebra of a Lie group~\cite{varadarajan1974lie}. More precisely, let $Z(X,Y)$ be the solution to $e^X e^Y = e^Z$, then 
\begin{equation}%
	\label{eq:BCH}
    Z(X,Y) = X + Y + \frac{1}{2} \left[X,Y\right] + \frac{1}{12} \left[ X, \left[X,Y\right]\right] +\frac{1}{12} \left[ Y, \left[Y,X\right]\right] + \ldots
\end{equation}
where ``\ldots'' indicates terms involving higher commutators of $X$ and $Y$.

Applying~\eqref{eq:BCH} and the property that $[X_f,X_g]=-X_{\jacobi{f,g}}$ to the product of exponentials in~\eqref{secondintegrator}, 
we obtain, up to fourth order in $\tau$:
\begin{align}
\exp\left(\frac{\tau}{2} \frac{\partial}{\partial t} \right) \exp\bigg(
    & (X_A + X_B + X_C)\tau \notag\\
    & + \frac{\tau^3}{12} \Big(X_{\jacobi{A,\jacobi{A,B}}} +X_{\jacobi{A,\jacobi{A,C}}} + X_{\jacobi{B,\jacobi{B,C}}} \notag\\
    &\qquad\quad
        + X_{\jacobi{A,\jacobi{B,C}}}
        +X_{\jacobi{B,\jacobi{A ,C}}}
    \Big) \notag\\
    & -\frac{\tau^3}{24} \Big(X_{\jacobi{B,\jacobi{B,A}}} + X_{\jacobi{C,\jacobi{C,A}}}+X_{\jacobi{C,\jacobi{C,B}}} \Big) \notag \\
    &+ \mathcal{O}(\tau^4) \bigg)\exp\left(\frac{\tau}{2} \frac{\partial}{\partial t} \right). \label{step1err}
\end{align}
Using the property that $X_{f+g} = X_f + X_g$
and the time reversibility,
we see that~\eqref{step1err} can be reduced to
\begin{equation}
\exp\left(\frac{\tau}{2} \frac{\partial}{\partial t} \right) \exp\bigg(\tau \left(X_{\mathcal{H} + \tau^2 \Delta \mathcal{H}} + \mathcal{O}(\tau^4) \right) \bigg)\exp\left(\frac{\tau}{2} \frac{\partial}{\partial t} \right), \label{timeind}
\end{equation}
where
\begin{align}
\Delta \mathcal{H} := \frac{1}{12} \bigg(&{\jacobi{A,\jacobi{A,B}}} +{\jacobi{A,\jacobi{A,C}}} + {\jacobi{B,\jacobi{B,C}}} \notag \\
&+{\jacobi{A,\jacobi{B,C}}}
    +{\jacobi{B,\jacobi{A,C}}}
    \notag \\
&-\frac{1}{2}{\jacobi{B,\jacobi{B,A}}} -\frac{1}{2}{\jacobi{C,\jacobi{C,A}}} - \frac{1}{2}{\jacobi{C,\jacobi{C,B}}} \bigg)
\end{align}
is the correction to the original Hamiltonian up to order two for the autonomous case.

Now we want to compute the modified Hamiltonian for a time--dependent system. The previous steps continue to hold, but we further need to include the time dependence applying again the BCH formula in~\eqref{timeind}.
Using the property
\begin{equation}
    	\left[\frac{\partial}{\partial t}, X_{\mathcal{H}}\right] =
    	X_{\frac{\partial \cH}{\partial t}}\,,
\end{equation}
we find that the
modified Hamiltonian $\tilde{\mathcal{H}}'$ is given by
\begin{equation}
	\tilde{\mathcal{H}}' := \mathcal{H} + \tau^2\underbrace{\left[ \Delta \mathcal{H} + \frac{1}{12} \jacobilr{\mathcal{H}, \frac{\partial \mathcal{H}}{\partial t}} - \frac{1}{24} \frac{\partial^2 \mathcal{H}}{\partial t^2}\right]}_{\Delta \cH ' }. \label{eq:modifiedhamiltonian}
\end{equation}

For Hamiltonians of the type~\eqref{eq:mainCH}, taking $A=f(t)s$, $B=V(q,t)$ and $C=\frac{p^2}{2}$, we obtain the explicit form:
\begin{align}
    \Delta \cH' =
        - \frac{1}{12} \bigg(
            &\frac{\partial f(t)}{\partial t} \left(\frac{p^2}{2} - V(q,t)\right)
            -f(t)^2 \left(\frac{p^2}{2} + V(q,t)\right) \notag\\
            &+f(t) \left(\frac{\partial V(q,t)}{\partial t}-p \frac{\partial V(q,t)}{\partial q}\right)
            +p \left(\frac{p}{2} \frac{\partial^2 V(q,t)}{\partial q^2} + \frac{\partial^2 V(q,t)}{\partial q \partial t}\right)\notag\\
            &-\left(\frac{\partial V(q,t)}{\partial q}\right)^2
            + \frac 12 \frac{\partial^2 V(q,t)}{\partial t^2}
            + \frac s2 \frac{\partial^2 f(t)}{\partial t^2}
        \bigg)\label{modHamiltonian}.
\end{align}

\begin{remark}%
	\label{Rem:singular1}
Note that~\eqref{eq:modifiedhamiltonian} is a truncation after the second order in $\tau$ of the modified Hamiltonian, which is an asymptotic series. This is important to keep in mind, especially if some of the terms in~\eqref{modHamiltonian} contain a singularity (i.e.\@ negative order terms) in $t$. In that case the overall result is not a $\tau^2$ term, because in the first few steps we have $t \approx \tau$, so the singularity in $t$ leads to an order reduction in $\tau$. Similar order reductions will take place in the higher order terms of the modified Hamiltonian, which we have not written explicitly here. We will see an example of this  in Section~\ref{sec:laneemden}. 
\end{remark}

The modified equation is the formal differential equation
\begin{align}
\dot{f}(q,p,s) &= X_{\tilde{\cH}'} f(q,p,s) + \cO(\tau^3 ) \notag\\
 &= X_\cH f(q,p,s) + \tau^2 X_{\Delta \cH'} f(q,p,s) + \cO(\tau^3 )\label{serieserror}
\end{align}
generated by the modified Hamiltonian, where $f$ is any smooth function of the dynamical variables. It has the property that solutions to the modified equation, truncated at a certain order in $\tau$, interpolate discrete solutions up to an error of the same order in $\tau$.

\begin{proposition}%
	\label{errorpropagationprop}
	If the Hamiltonian does not contain any singularities in the time variable, then the integrator is of second order and the local error is 
	\begin{equation}%
	\label{eq:Prop10error}
	S_2(\tau)(q,p,s) - \varphi_\tau(q,p,s)
	=  \tau^3 X_{\Delta \cH'} (q,p,s) + \cO(\tau^4) ,
	\end{equation}
	where $\varphi_\tau(q_0,p_0,s_0)$ denotes the exact flow after time $\tau$ and $S_2(\tau)$ the numerical integrator.
\end{proposition}
\begin{proof}
	We have that
\begin{align*}
S_2(\tau)(q,p,s) - \varphi_\tau(q,p,s)
&= \exp \left( \tau X_{\tilde{\cH}' + \cO(\tau^3) } (q,p,s)- \tau X_{\cH} (q,p,s) \right) \\
&= \tau^3 X_{\Delta \cH'} (q,p,s) + \cO(\tau^4). 
\qedhere
\end{align*}
\end{proof}

In particular, if $(q,p,s)$ are Darboux coordinates, we find as local error estimates in each of the coordinates:
\begin{align}
\Delta q^a &=
\tau^3 | X_{\Delta \cH'} q^a| = \tau^3 \left|\frac{\partial \Delta \cH'}{\partial p_a}\right| \,, \label{errqstep1}\\
\Delta p_a &=
\tau^3 | X_{\Delta \cH'} p_a| = \tau^3\left|-\frac{\partial \Delta \cH'}{\partial q^a} - p_a\frac{\partial \Delta \cH'}{\partial s}\right|\,, \label{errpstep1}\\
\Delta s &= 
\tau^3 | X_{\Delta \cH'} s|\, = \tau^3\left| \frac{\partial \Delta \cH'}{\partial p_a}p_a - \Delta \cH'\right|. \label{errsstep1}
\end{align}
\noeqref{errqstep1}\noeqref{errpstep1}\noeqref{errsstep1}
Consider a numerical solution $(q_j,p_j,s_j) = S(\tau)^j (q_0,p_0,s_0)$ and an exact solution $(q(t),p(t),s(t)) = \varphi_{t}(q_0,p_0,s_0)$ with the same initial data. We can estimate an upper bound for the error $\Delta_j =
\left\| (q_j,p_j,s_j) - (q(j\tau),p(j\tau),s(j\tau)) \right\|$ after $j$ steps by
\begin{align*}
\Delta_{j+1}
&= \left\| (q_{j+1},p_{j+1},s_{j+1}) - (q((j+1)\tau),p((j+1)\tau),s((j+1)\tau)) \right\| \\
&= \left\| S(\tau)(q_j,p_j,s_j) - \varphi_{\tau} ( (q(j\tau),p(j\tau),s(j\tau)) ) \right\|  \\
&\leq  \left\| S(\tau)(q_j,p_j,s_j) - \varphi_{\tau} (q_j,p_j,s_j) \right\| \\
    &\quad+ \left\| \varphi_{\tau} (q_j,p_j,s_j) - \varphi_{\tau} ( (q(j\tau),p(j\tau),s(j\tau)) ) \right\| \\
&= \tau^3 \left\| X_{\Delta \cH'} (q_j,p_j,s_j) \right\| + \cO(\tau^4) + \left\| \nabla \varphi_{\tau} \Delta_j \right\| + \cO(\Delta_j^2),
\end{align*}
where $ \nabla \varphi_{\tau} = I + \cO(\tau)$ because any integrator is close to the identity map. Hence as long as the error is small, $\Delta_j = \cO(\tau^3)$, we have
\[ \Delta_{j+1} \leq \Delta_j + \tau^3 \left\| X_{\Delta \cH'} (q_j,p_j,s_j) \right\| + \cO(\tau^4) \]
which gives us an estimate for the error after $N$ steps:
\[ \Delta_N \lesssim \sum_{j=0}^{N-1} \tau^3 \left\| X_{\Delta \cH'} (q_j,p_j,s_j) \right\|. \]

\begin{remark}%
	\label{rmk:errorsize}
	In the proof above we obtain an upper bound for the numerical error in an asymptotic sense. 
	For relatively large $\tau$ the $\cO(\tau^4)$ term will not be negligible. In addition, after several integration steps, $\Delta_j$ will likely be too large, violating the assumption that $\Delta_j = \cO(\tau^3)$.
	This can be seen in a few instance of the examples below.
\end{remark}
\begin{example}
	The contact Hamiltonian of a damped harmonic oscillator is given by~\cite{bravetti2017contact}
	\begin{equation}
		\cH = \underbrace{\frac{1}{2} p^2}_{C} + \underbrace{\frac{1}{2} q^2}_{B} + \underbrace{\alpha s}_{A},\quad \alpha\in\mathbb{R}. \label{dampedharmonic}
	\end{equation}
	From equation~\eqref{eq:modifiedhamiltonian} it follows readily that
	\begin{equation}
    \tilde{\cH}' = \frac{1}{2} p^2 + \frac{1}{2} q^2+ \alpha s -
    \frac{\tau^2}{24} \bigg(2\alpha p q + p^2(\alpha^2-1) + q^2 (\alpha^2 + 2)\bigg)
	\end{equation}
	In this case Proposition~\ref{errorpropagationprop} clearly holds, and it implies that the errors in the kinematic quantities at each step
    are
	\begin{align}
		\Delta \dot{q}_i &= \frac{\tau^2}{12} \left|\alpha q_{i-1} + p_{i-1}(\alpha^2-1) \right| \label{harmonicerrorq} \\
		\Delta \dot{p}_i &= \frac{\tau^2}{12} \left|q_{i-1}(\alpha^2 + 2) + \alpha p_{i-1}\right| \label{harmonicerrorp}\\
		\Delta \dot{s}_i &= \frac{\tau^2}{24} \left|p_{i-1}^2 (\alpha^2-1)-q_{i-1}^2 (\alpha^2 + 2)\right|.\label{harmonicerrors}
	\end{align}
	\noeqref{harmonicerrorp}\noeqref{harmonicerrorq}\noeqref{harmonicerrors}
	For $\tau$ small enough,
	 we obtain from~\eqref{harmonicerrorq} that
	\begin{equation}%
	\label{eq:errorDamped}
		\Delta q_{i+1} \simeq \Delta q_{i} + \frac{\tau^3}{12} \left| \alpha q_{i-1} + p_{i-1}(\alpha^2-1) \right|.
	\end{equation}
	In this example we can compute the exact solution for $q(t)$ and compare the numerical error at each step, labelled \emph{Numerical}, 
	with the estimate provided by equation~\eqref{eq:errorDamped} using the modified Hamiltonian, labelled \emph{Estimated} 
	(see Figures~\ref{fig:errorplotharmonic05} and~\ref{fig:errorplotharmonic0001}).

	\begin{figure}[ht!]
        \centering
		\includegraphics[width=\onelinewidth, trim=0 0 0 35, clip]{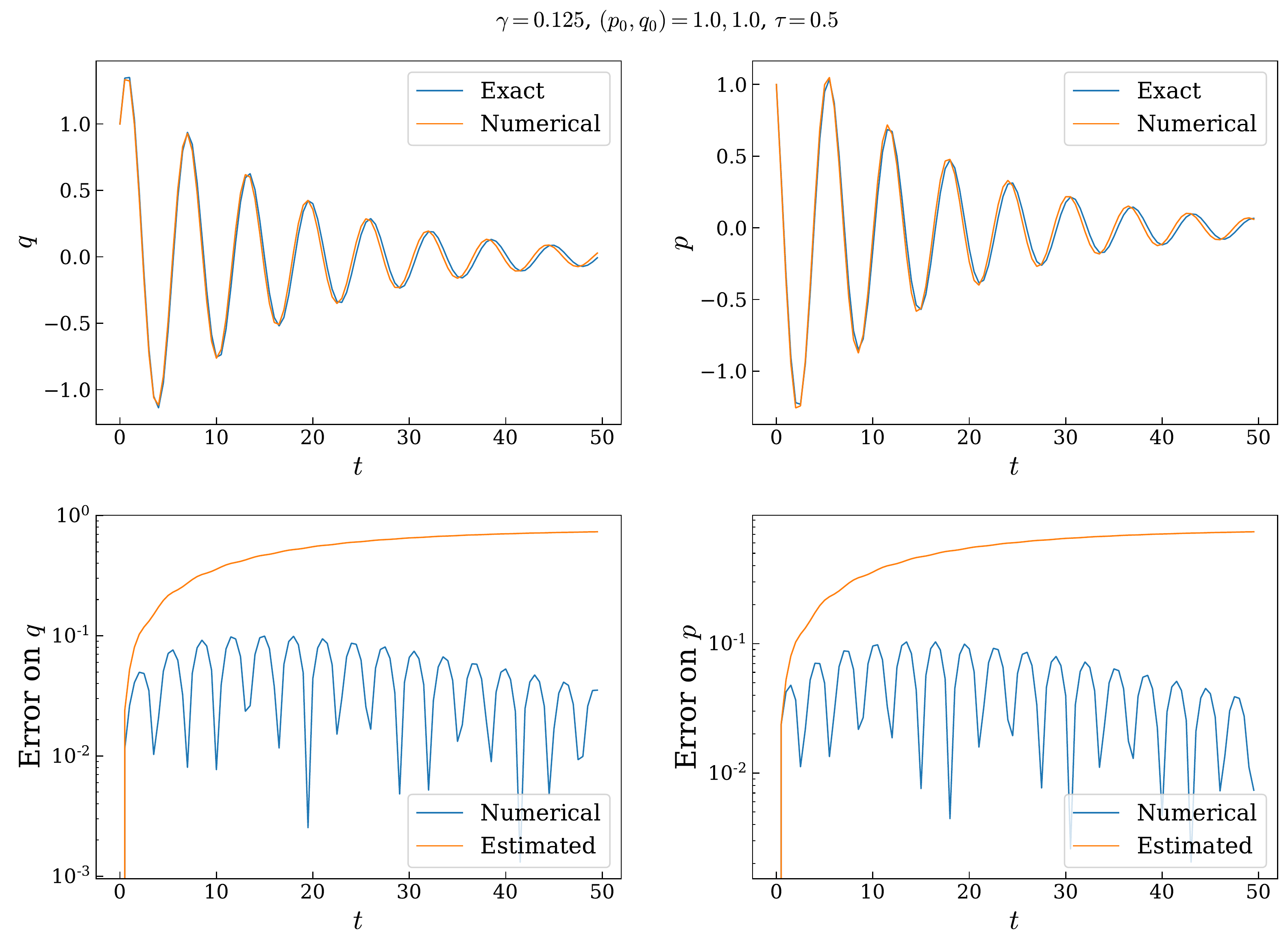}
		\captionsetup{font=footnotesize}
		\captionsetup{width=\onelinewidth}
		\caption{Numerical error and error estimate using the modified Hamiltonian for the damped harmonic oscillator \eqref{dampedharmonic} with coupling parameter $\alpha=0.125$ and time step $\tau = 0.5$.}%
	\label{fig:errorplotharmonic05}
	\end{figure}
    \begin{figure}[ht!]
        \centering
		\includegraphics[width=\onelinewidth, trim=0 0 0 35, clip]{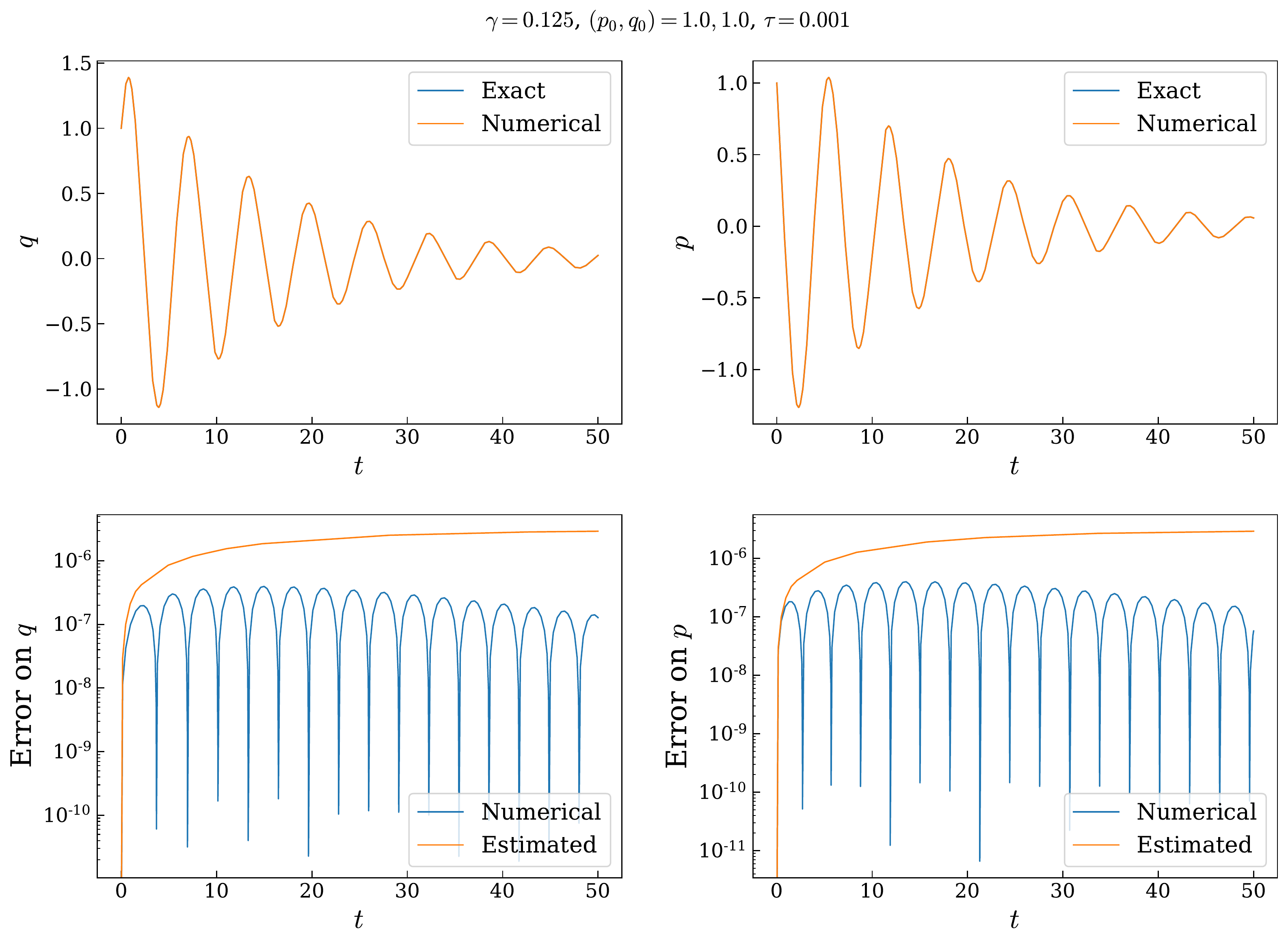}
		\captionsetup{font=footnotesize}
		\captionsetup{width=\onelinewidth}
		\caption{Numerical error and error estimate using the modified Hamiltonian for the damped harmonic oscillator \eqref{dampedharmonic} with coupling parameter $\alpha=0.125$ with time step $\tau = 0.001$.}%
	\label{fig:errorplotharmonic0001}
	\end{figure}
\end{example}

\section{Contact integrators in celestial mechanics}%
	\label{sec:examples}
\subsection{The perturbed Kepler problem}

The most general form of a perturbed Kepler problem is~\cite{fonda2012periodic}
\begin{equation}
    \ddot{x} + \frac{x}{|x|^3} = \alpha F(x, \dot{x}, t; \alpha).
\end{equation}
Here, the perturbation $F(x, \dot{x}, t; \alpha)$ is usually assumed to be either of the form $F=\partial V(x, t)/\partial x$, where $V(x,t)$ is a periodic function in $t$, or $F$ may include different types of dissipation, the simplest one being a linear drag~\cite{breiter1998unified,diacu1999two,margheri2014dynamics,pub.1059182093}.  
In the case of a linear drag, the corresponding equation 
        \begin{equation}%
	\label{eq:modifiedKeplerO}
         \ddot{x} +\alpha\dot x+ \frac{x}{|x|^3} = 0.
         \end{equation}
is obviously of the form~\eqref{eq:geneq}, with 
     \begin{equation}%
	\label{eq:cHKeplerO}
        \cH=\frac{|p|^2}{2}-\frac{1}{|x|}+\alpha\,s\,.
    \end{equation}
We refer to~\cite{margheri2014dynamics} for a detailed analysis of the dynamics in this case.
Here, to show the usefulness of our integrators, we 
study the slightly more general case of a linear drag that also depends explicitly on time.
To the best of our knowledge, this case has not been addressed before.
The goal in this section
is not to give a complete study of the orbits of the system, but to compare the behavior of contact integrators with respect to standard (fixed--step Runge--Kutta) numerical methods.

The equation for the modified Kepler problem that we consider is 
    \begin{equation}%
	\label{eq:modifiedKepler}
         \ddot{x} +\alpha \sin(\Omega t)\,\dot x+ \gamma \frac{x}{|x|^3} = 0, \quad \alpha,\Omega,\gamma\in\mathbb{R}.
    \end{equation}
Clearly equation~\eqref{eq:modifiedKepler} is of the type~\eqref{eq:geneq}, the corresponding contact Hamiltonian being 
    \begin{equation}%
	\label{eq:cHKepler}
        \cH=\frac{|p|^2}{2}-\frac{\gamma}{|x|}+\alpha \sin(\Omega t)\,s\,.
    \end{equation}

\subsubsection{Error analysis}

The perturbed Kepler problem, and in particular the Hamiltonian~\eqref{eq:cHKepler}, satisfies the hypothesis of Proposition~\ref{errorpropagationprop}, and therefore we are allowed to estimate the error at each step by using expression~\eqref{eq:Prop10error}, with the modified Hamiltonian being
\begin{align}
    \Delta \cH = \frac1{{24 \left| q\right| ^5}} \bigg(
        &\alpha  \left| p\right| ^2 \left| q\right| ^5 \left(\alpha  \sin ^2(t \omega )-\omega  \cos (t \omega )\right)
         +2 \alpha  \left| q\right| ^2 p\cdot q \sin (t \omega ) \notag \\
        &+\alpha  s \omega ^2 \left| q\right| ^5 \sin (t \omega )
         -2 \alpha \left| q\right| ^4 \left(\alpha  \sin ^2(t \omega )+\omega  \cos (t \omega )\right) \notag \\
        &+2 \left| q\right| 
         +2 (p\cdot q)^2
         -p_1^2 q_2^2
         -p_2^2 q_1^2
         +2 p_1 p_2 q_1 q_2 
        \bigg),
    \label{modifiedkepler}
\end{align}
where
\begin{equation}
    \left|q\right|=q_1^2+q_2^2, \qquad \left|p\right|=p_1^2+p_2^2\,.
\end{equation}

Since the explicit expressions for the estimated errors are quite cumbersome and not particularly illuminating, we omit them in the text. However, in Figures~\ref{fig:pkeplererr0},~\ref{fig:pkeplererr1} and~\ref{fig:pkeplererr2} we report a comparison between the error estimate via the modified Hamiltonian for the second order integrator and the numerical error. The latter was calculated by comparing the trajectory obtained using the second order integrator to a much more accurate numerical solution obtained using a sixth order integrator. As the figures show, the error estimated by means of the modified Hamiltonian matches the numerical error quite accurately. However, the estimate fails to be an upper bound in the error, as was anticipated in Remark~\ref{rmk:errorsize}.

\begin{figure}[ht!]
    \centering
    \includegraphics[width=\onelinewidth]{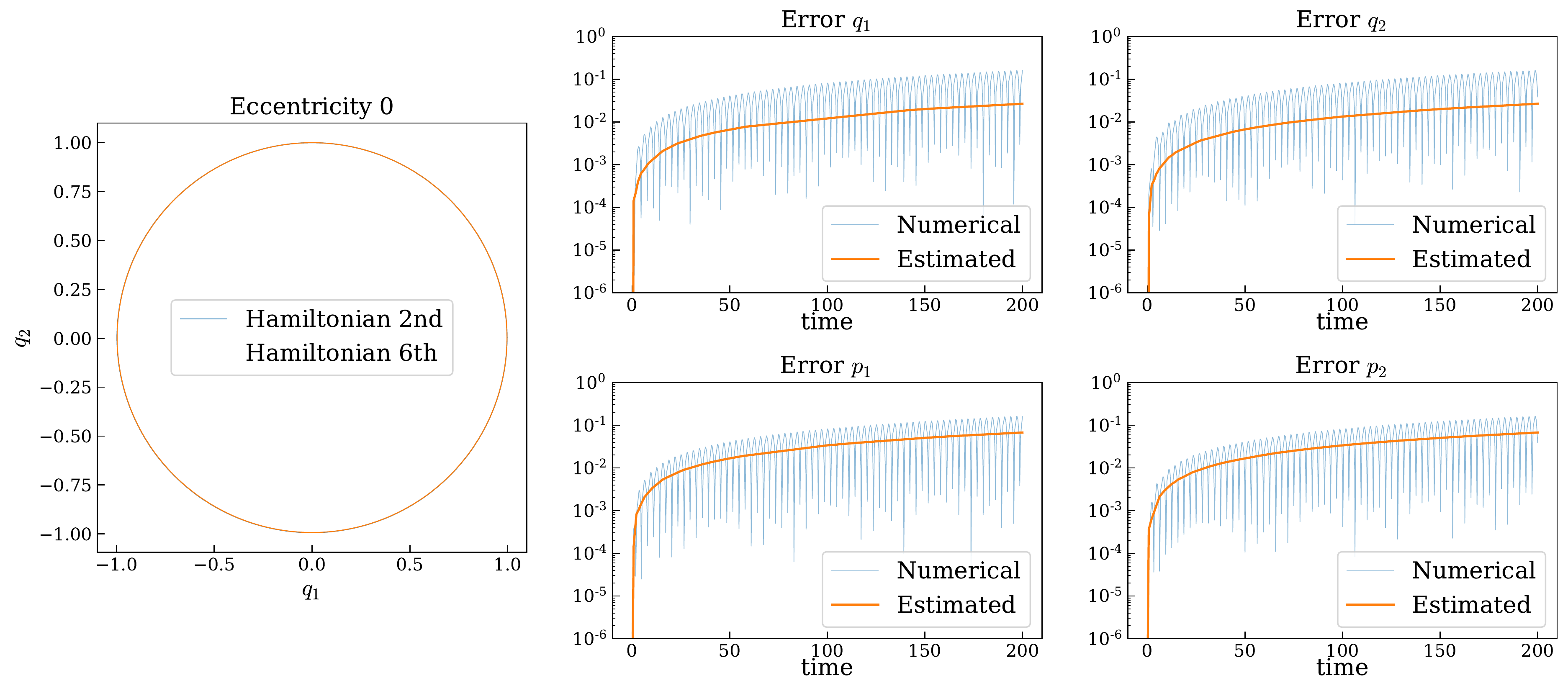}
    \captionsetup{width=\onelinewidth}
    \captionsetup{font=footnotesize}
    \caption{Error control for an integrated orbit of the perturbed Kepler problem \eqref{eq:cHKepler} with $\tau=0.05$, eccentricity $0$, $\Omega = 2\pi$, $\alpha=-0.01$ and $\gamma = 1$. We have denoted $q=(q_1,q_2)$ and $p=(p_1,p_2)$.}%
	\label{fig:pkeplererr0}
\end{figure}

\begin{figure}[ht!]
    \centering
    \includegraphics[width=\onelinewidth]{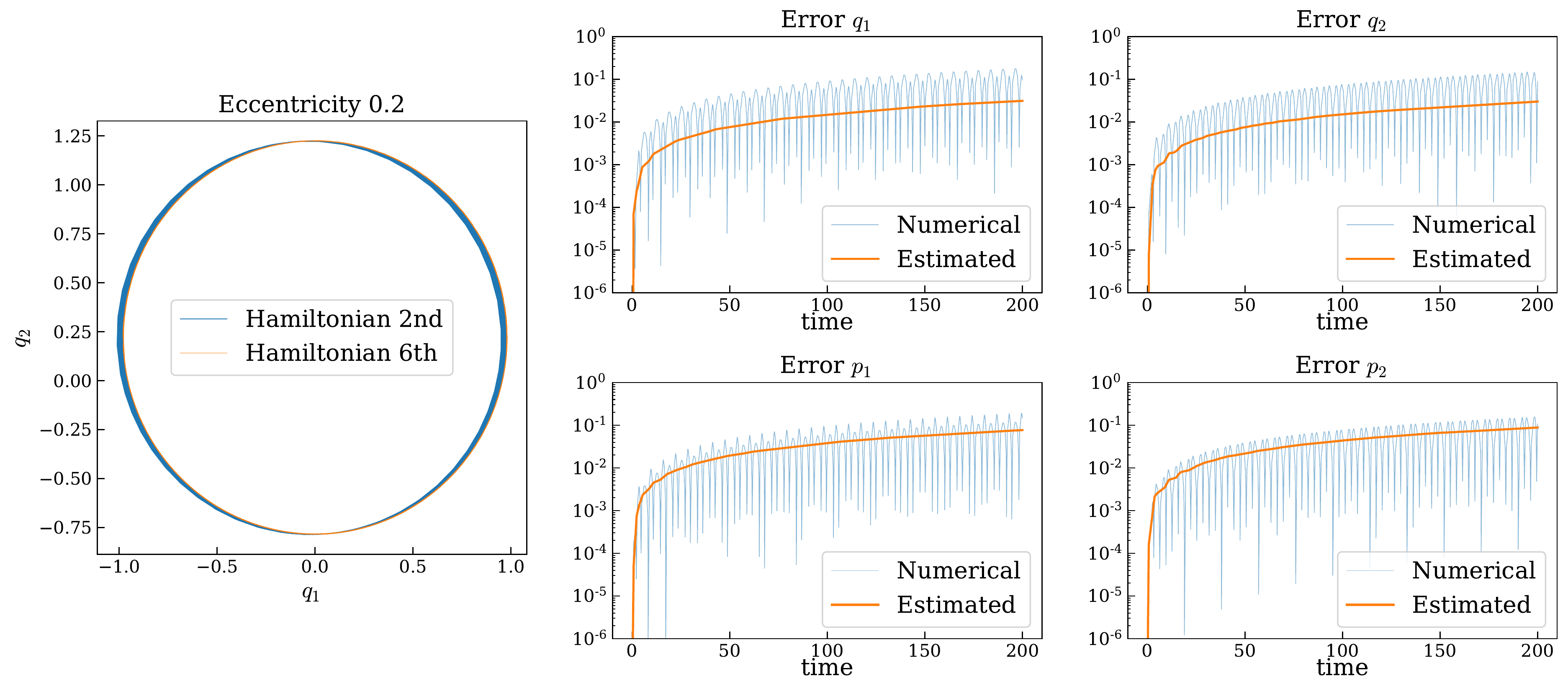}
    \captionsetup{width=\onelinewidth}
    \captionsetup{font=footnotesize}
    \caption{Error control for an integrated orbit of the perturbed Kepler problem \eqref{eq:cHKepler} with $\tau=0.05$, eccentricity $0.2$, $\Omega = 2\pi$, $\alpha=-0.01$ and $\gamma = 1$. We have denoted $q=(q_1,q_2)$ and $p=(p_1,p_2)$.}%
	\label{fig:pkeplererr1}
\end{figure}

\begin{figure}[ht!]
    \centering
    \includegraphics[width=\onelinewidth]{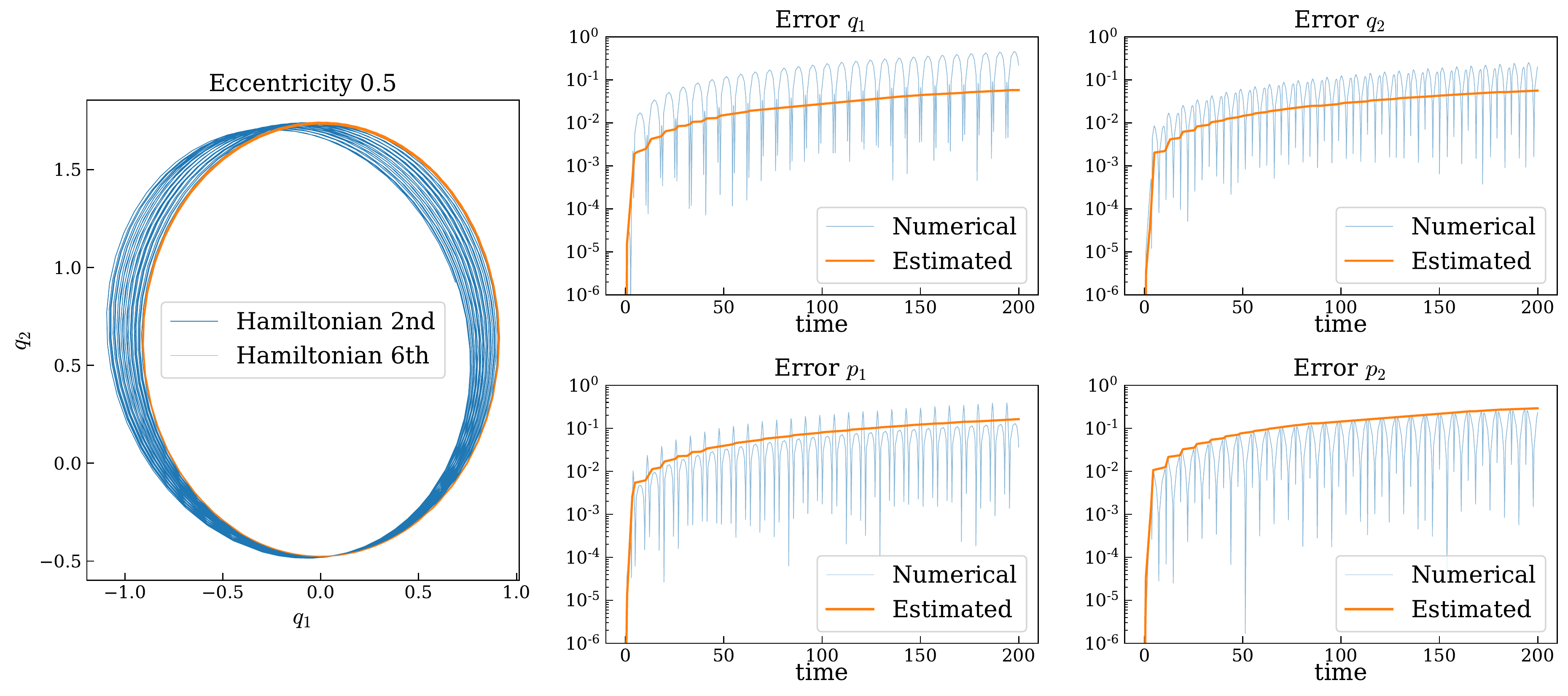}
    \captionsetup{width=\onelinewidth}
    \captionsetup{font=footnotesize}
    \caption{Error control for an integrated orbit of the perturbed Kepler problem \eqref{eq:cHKepler} with $\tau=0.05$, eccentricity $0.5$, $\Omega = 2\pi$, $\alpha=-0.01$ and $\gamma = 1$. We have denoted $q=(q_1,q_2)$ and $p=(p_1,p_2)$.}%
	\label{fig:pkeplererr2}
\end{figure}

\subsubsection{Numerical results}

In all the experiments presented in this section, we keep the value of $\alpha$ small, so that the trajectories have a structure comparable to the standard Kepler problem and are easier to discuss and interpret.
When using the Hamiltonian integrator of the sixth order we specify with A, B, C or E if we use, respectively, the approximate coefficients from column A, B, C of Table~\ref{numsol} or the exact coefficients from~\eqref{2n+2orderexact}.
The second order variational integrator is the one from~\cite{vermeeren2019contact}, while the fourth order variational integrator is the one introduced in Section~\ref{sec:Lagrangian}.

In all figures, the eccentricity of the trajectory is defined in the following sense: it is the eccentricity of the trajectory of an unperturbed Kepler problem with the same initial conditions.

In this example, our integrators show a remarkable stability at both second, fourth and sixth order for very large time step, 
when compared against a fixed--step Runge--Kutta method of fourth order (RK4 in the plots).

The Kepler problem is well--suited to emphasize the differences between our algorithms. As one expects, for large time steps, the solution obtained from the Runge--Kutta integrator drifts toward the singularity and explodes in a rather short amount of iterations (see Figure~\ref{fig:pkeplerstable}).
More interestingly, depending on the value of $\tau$ and the ellipticity of the trajectory, we can observe a remarkable difference in the behaviour of our contact integrators: the Hamiltonian integrators, even of high order, tend to show a more pronounced precession compared to the variational ones; on the other hand the former seem to be much more stable than the latter when the time step increases. See Figures~\ref{fig:pkeplercompare001} and~\ref{fig:pkeplercompare04}. Note that the different stability observed could be due to lower order of the variational integrator.

Further figures showing the total error and the error on the contact Hamiltonian function can be found with the simulation notebooks in~\cite{zenodoCode}. We decided to omit them from the paper as we did not find them clearer than the direct comparison of the solutions.

\begin{figure}[p]
    \centering
    \includegraphics[width=\onelinewidth]{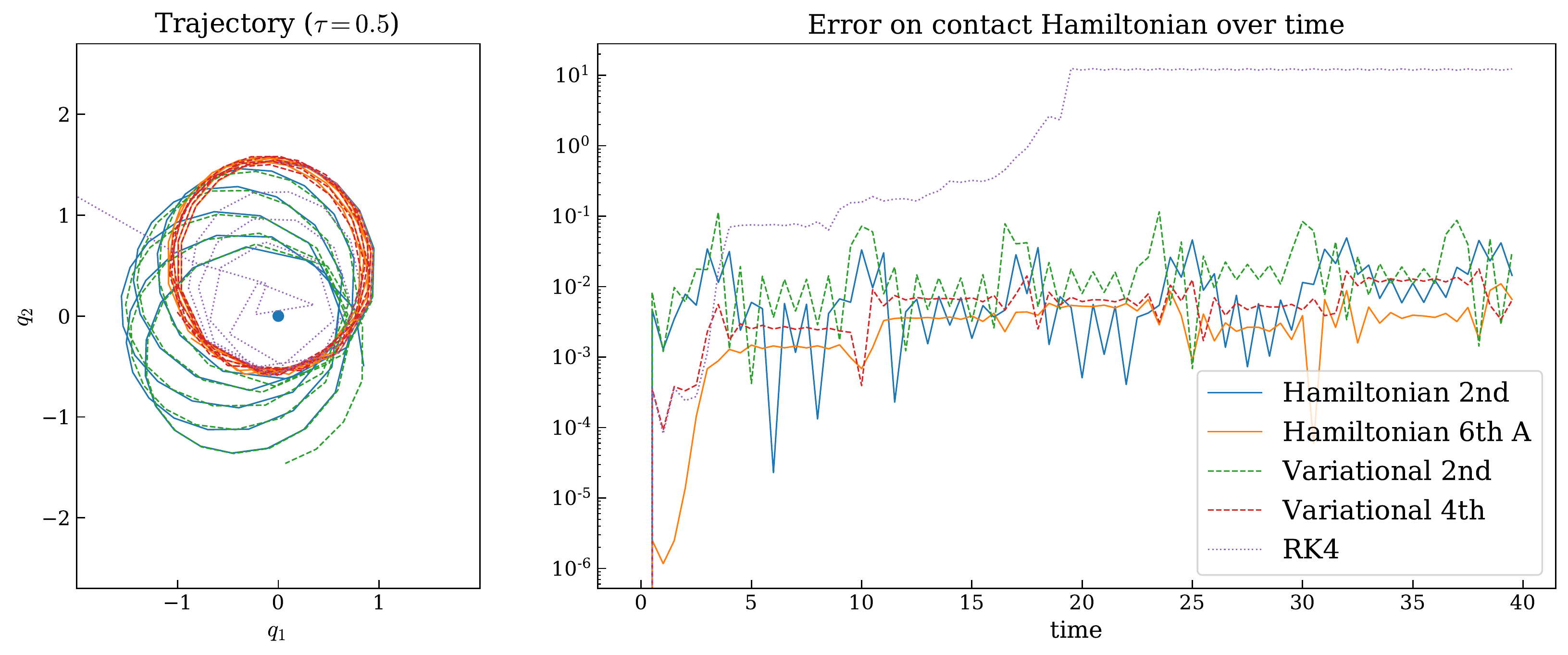}
    \captionsetup{width=\onelinewidth}
    \captionsetup{font=footnotesize}
    \caption{Integrated orbits with eccentricity $0.4$ for the perturbed Kepler problem \eqref{eq:cHKepler} with $\Omega = \pi$, $\alpha=-0.07$ and $\gamma = 1$. Here $\tau=0.5$ and the trajectories are integrated in the time interval $t\in[0,40]$. We can see that Runge--Kutta is diverging very quickly, while the other integrators are all stable and show a different degree of precession.}%
    \label{fig:pkeplerstable}
\end{figure}
\begin{figure}[p]
    \centering
    \includegraphics[width=\onelinewidth]{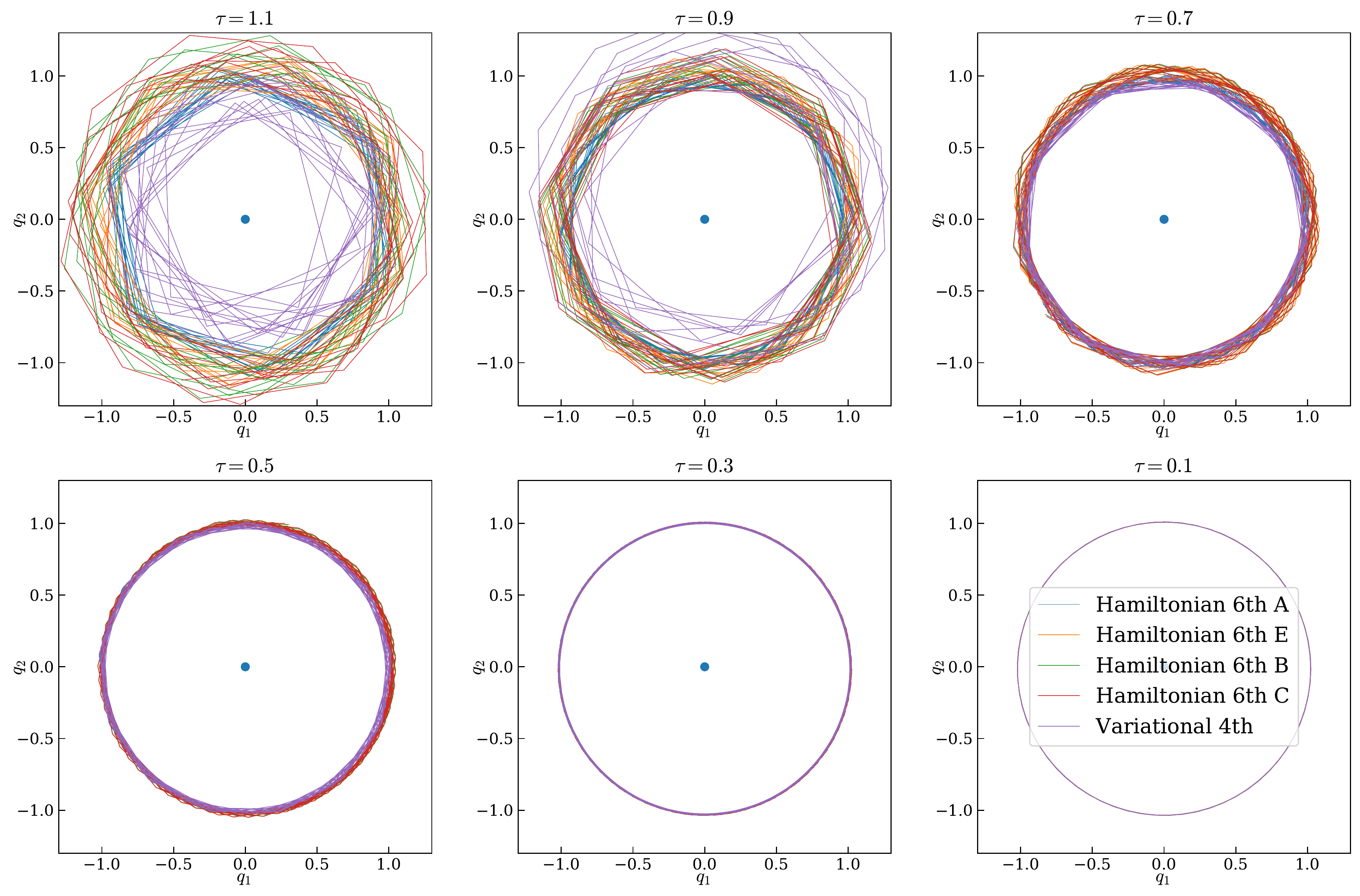}
    \captionsetup{width=\onelinewidth}
    \captionsetup{font=footnotesize}
    \caption{Integrated orbits with eccentricity $0.01$ for the perturbed Kepler problem \eqref{eq:cHKepler} with $\Omega = 2\pi$, $\alpha=-0.07$ and $\gamma = 1$. Here we let $\tau$ vary and integrate the trajectories in the time interval $t\in[0,80]$. We can see that for small eccentricity all the integrators are stable, while the Hamiltonian ones have better properties for large time steps.
    }%
    \label{fig:pkeplercompare001}
\end{figure}
\begin{figure}[ht!]
    \centering
    \includegraphics[width=\onelinewidth]{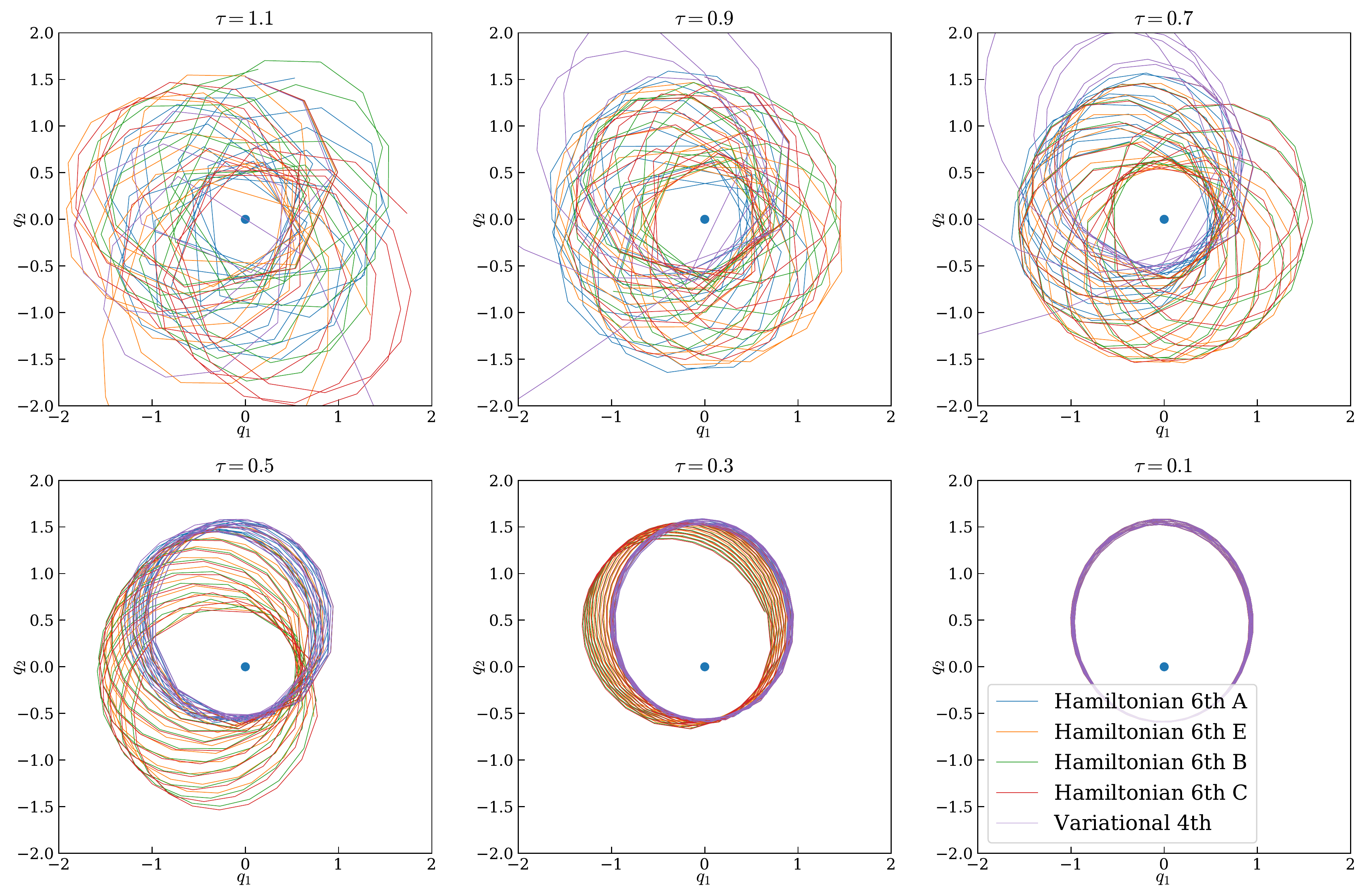}
    \captionsetup{width=\onelinewidth}
    \captionsetup{font=footnotesize}
    \caption{Integrated orbits with eccentricity $0.4$ for the perturbed Kepler problem \eqref{eq:cHKepler} with $\Omega = 2\pi$, $\alpha=-0.07$ and $\gamma = 1$. Here we let $\tau$ vary and integrate the trajectories in the time interval $t\in[0,80]$. In this case the variational integrator shows its limits for large time steps, while the Hamiltonian integrator remains very stable for all time steps, at the price of a more pronounced precession.}%
    \label{fig:pkeplercompare04}
\end{figure}

\subsection{The spin--orbit model}
In this section we consider the so--called spin--orbit model in the version presented in~\cite{gkolias2017hamiltonian}, trying to use the same notation as in the referenced paper as much as possible.

This model describes the motion of a small body, e.g.~a satellite, that moves around a larger body on a Keplerian orbit and rotates around its shorter principal axis with zero obliquity (see also~\cite{Celletti1990}). The corresponding Newton equation is of the form~\eqref{eq:geneq} and is given by a second--order time--dependent differential equation in the angle that describes the relative orientation of the 
longer principal axis with respect to a preassigned direction. The time variations in the moment of inertia of the satellite introduce an angular velocity--dependent term that accounts for the body’s rotation in addition to the external torques.

More precisely, the equation
    \begin{equation}%
	\label{eq:spinorbit1}
        \frac{d \Gamma}{dt} = \frac{dC}{dt} \dot\theta + C \ddot\theta = N_z(\theta, t)
    \end{equation}
describes the rotation of the body around its principal axis, with moment of inertia $C$. In the equation, $\Gamma$ represents the angular momentum of the body, $\dot\theta$ the angular velocity and $N_z$ the external torques.

For $C\neq 0$, equation~\eqref{eq:spinorbit1} can be rewritten as
    \begin{equation}%
	\label{eq:spinorbit2}
       \ddot\theta +\frac{dC}{dt} \frac{\dot\theta}{C} - \frac{N_z(\theta, t)}{C}=0\,,
    \end{equation}
which is clearly of the type~\eqref{eq:geneq}, with contact Hamiltonian
    \begin{equation}%
	\label{eq:cHspinorbit}
        \cH=\frac{p^2}{2}+ \frac{N_z(\theta, t)}{C}+\frac{dC}{dt} \frac{1}{C}\,s\,.
    \end{equation}

In the examples that follow, as in~\cite{gkolias2017hamiltonian}, we will consider a moment of inertia that varies periodically around an average value $\widetilde C$ with frequency $\Omega$, namely
\begin{equation}%
	\label{eq:c-periodic}
    C(t) = \widetilde{C} + \lambda \cos(\Omega t),
\end{equation}
and we will focus on two particular forms of the torque:
\begin{itemize}
    \item The gravitational torque for a triaxial rigid body on a Keplerian elliptical orbit around a point perturber:
\begin{align}
N_{z}^{\mathrm{triaxial}}(\theta, t)
&= -\frac{3 \nu(B-A)}{2} \frac{\alpha}{r}^3 \sin(2\theta - 2f) \\
&= -\frac{3 \nu(B-A)}{2} \sum_{m\in\mathbb{Z}\setminus\{0\}} W\left(\frac m2, e\right) \sin(2\theta - mt)
\end{align}
where $A<B<C$ are the moments of inertia in the body frame, $\alpha$ the semi--major axis, $\nu$ the orbital frequency, $r$ the distance between the bodies, $f$ the true anomaly and $W(m/2, e)$ are the coefficients of the Fourier expansion w.r.t.\ the periodic functions $r$ and $t$. We refer to~\cite{Celletti1990} for a clear explanation of the terminology. Note in particular that the coefficients $W(m/2, e)$, called Cayley coefficients, are power series of the eccentricity: some of their values can be found in~\cite[Table 2.1]{Celletti1990} or~\cite[pp. 271--274]{Cayley}. In the examples we will truncate the series dropping all the powers of the eccentricity that give a contribution smaller than the error.

\item The torque from a third body perturbation:
\[
N_{z}^{\mathrm{tidal}}(\theta, t) = \mu + a \dot\theta
\]
where $(\mu, a)\in\mathbb{R}_+\times\mathbb{R}_-$.
\end{itemize}

In what follows we want to show that a direct application of our integrators allows us to recover the phase space plots given in~\cite{gkolias2017hamiltonian}, including the capture into a synchronous resonance, with great accuracy and much less effort compared to the algorithm in~\cite{gkolias2017hamiltonian},  additionally guaranteeing the preservation of the underlying geometric structure.

In the first two examples we will assume $a=\mu=0$ and thus no torque from a third body perturbation is present. In any case, we will get a system of the form~\eqref{eq:geneq} with
\begin{align*}
    f(t) &= -a-\frac{\lambda\Omega \sin(\Omega t)}{\widetilde{C} + \lambda \cos(\Omega t)}\\
    \frac{\partial V(q,t)}{\partial q} &= -\mu + \frac{3}{2} \frac{\nu(B-A)}{\widetilde{C} + \lambda \cos(\Omega t)} \sum_{m\in\mathbb{Z}\setminus\{0\}} W\left(\frac m2, e\right) \sin(2\theta - mt).
\end{align*}

Before concluding, we would like to remark that the spin-orbit model falls in the class of models satisfying Proposition~\ref{errorpropagationprop}: it is thus possible to have a rather precise control on the error {by means of the modified Hamiltonian}. However, this is out of the scope of this section and we will not pursue it further.

\subsubsection{Numerical Results}

In Figures~\ref{fig:spinorbit} and~\ref{fig:spinorbit5} we can see the plots of the stroboscopic surface of section obtained by slicing the trajectories at times that are multiples of $2\pi$. These correspond respectively to~\cite[Figure 1 and Figure 5]{gkolias2017hamiltonian}.

For all the figures of this section we use $\widetilde C = 1$ and $\nu = 1$.
In Figure~\ref{fig:spinorbit}, depicting the resonant case $\Omega=1$, we use $B-A = e = 0.01$. In Figure~\ref{fig:spinorbit5}, depicting the non--resonant case $\Omega=\sqrt{2}$, we use $B-A = e = 0.04$.
In both cases the only external torque is the triaxal, i.e. $a = \mu = 0$.

\begin{figure}[ht!]
    \centering
    \includegraphics[width=\halflinewidth]{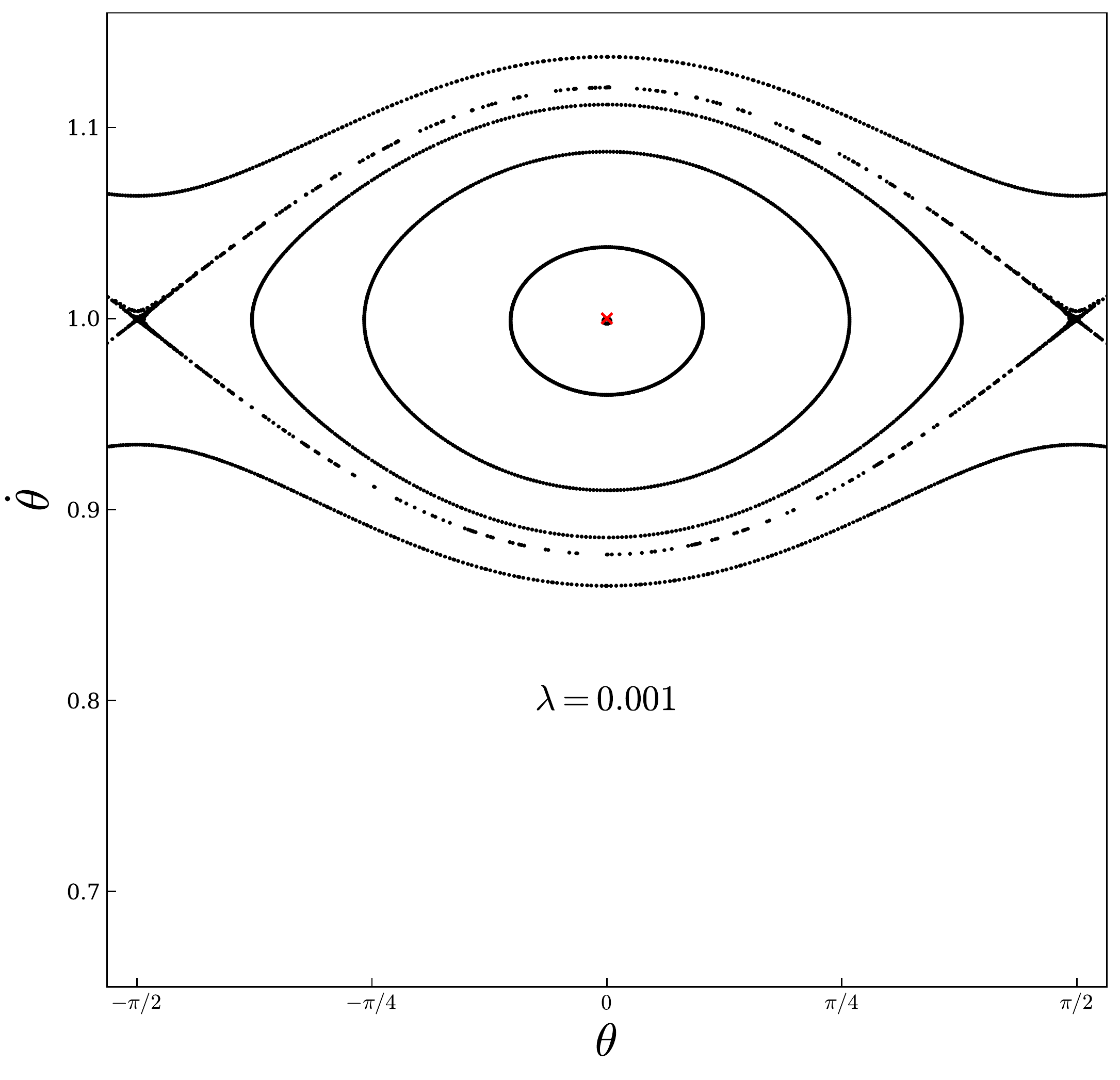}
    \includegraphics[width=\halflinewidth]{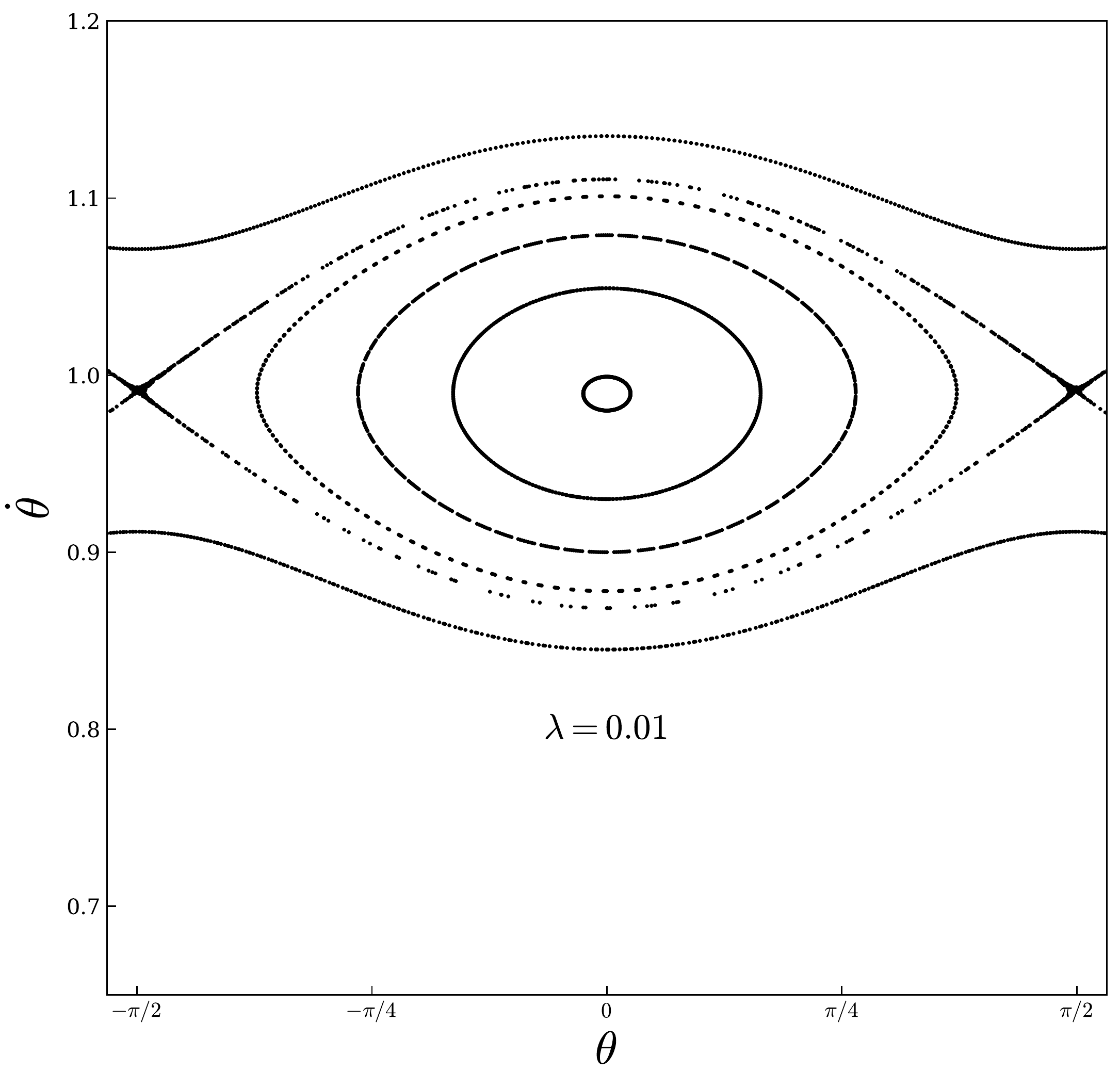}
    \includegraphics[width=\halflinewidth]{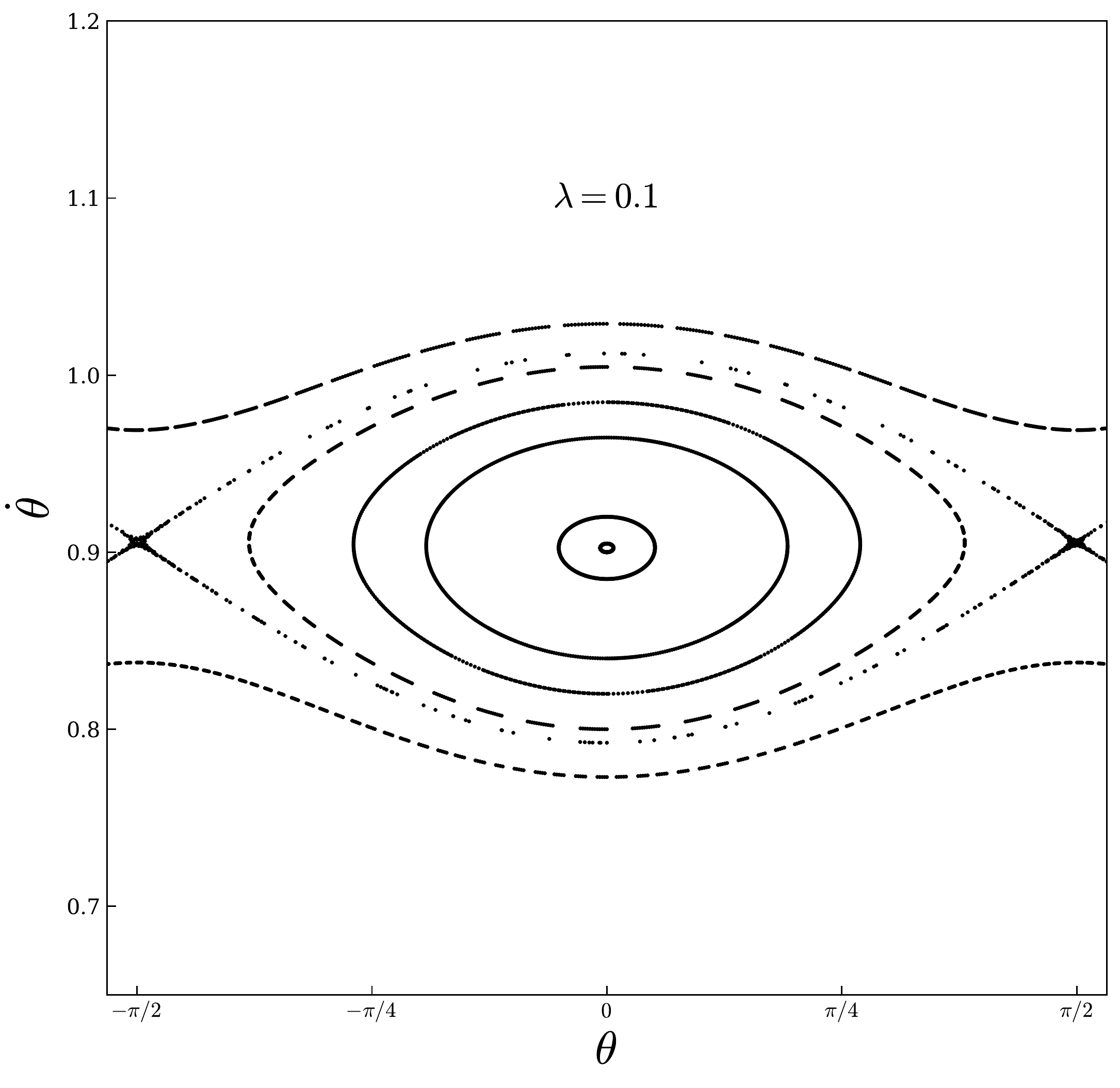}
    \includegraphics[width=\halflinewidth]{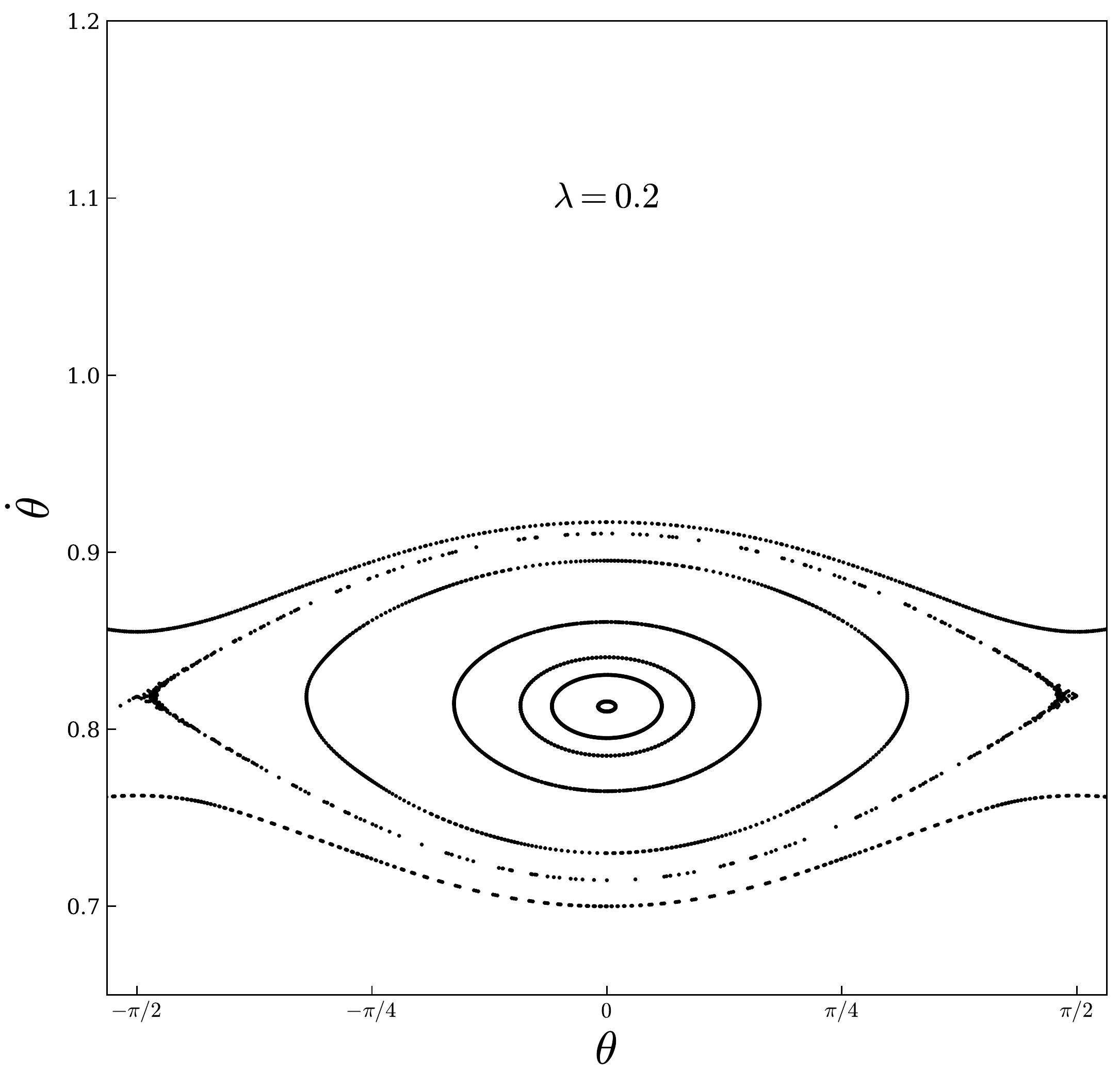}
    \captionsetup{width=\onelinewidth}
    \captionsetup{font=footnotesize}
    \caption{Poincar\'e surfaces of section for the spin-orbit problem \eqref{eq:cHspinorbit} in the resonant case $\Omega=1$.
             The coupling constant $\lambda$ is indicated in each graph.}%
    \label{fig:spinorbit}
\end{figure}

\begin{figure}[ht!]
    \centering
    \includegraphics[width=\onelinewidth]{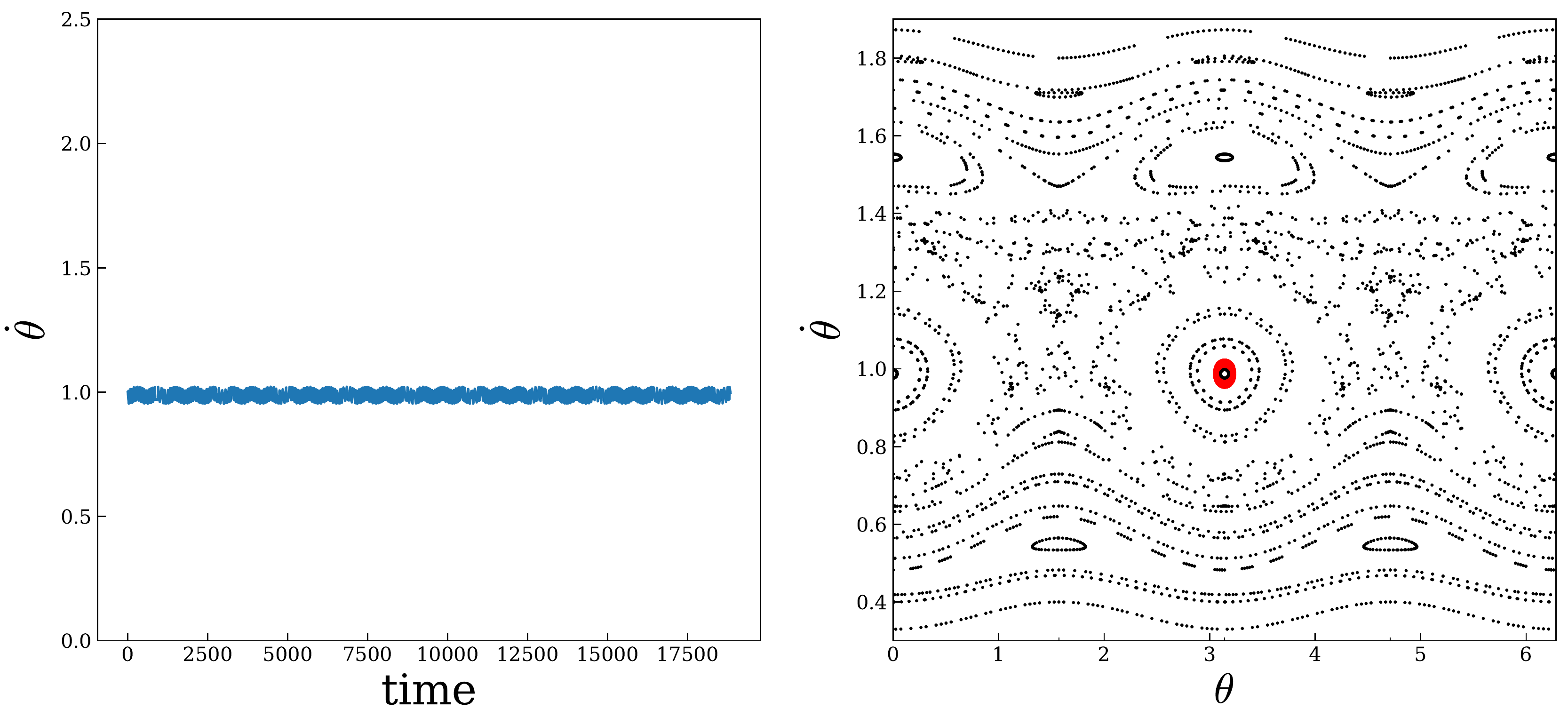}
    \includegraphics[width=\onelinewidth]{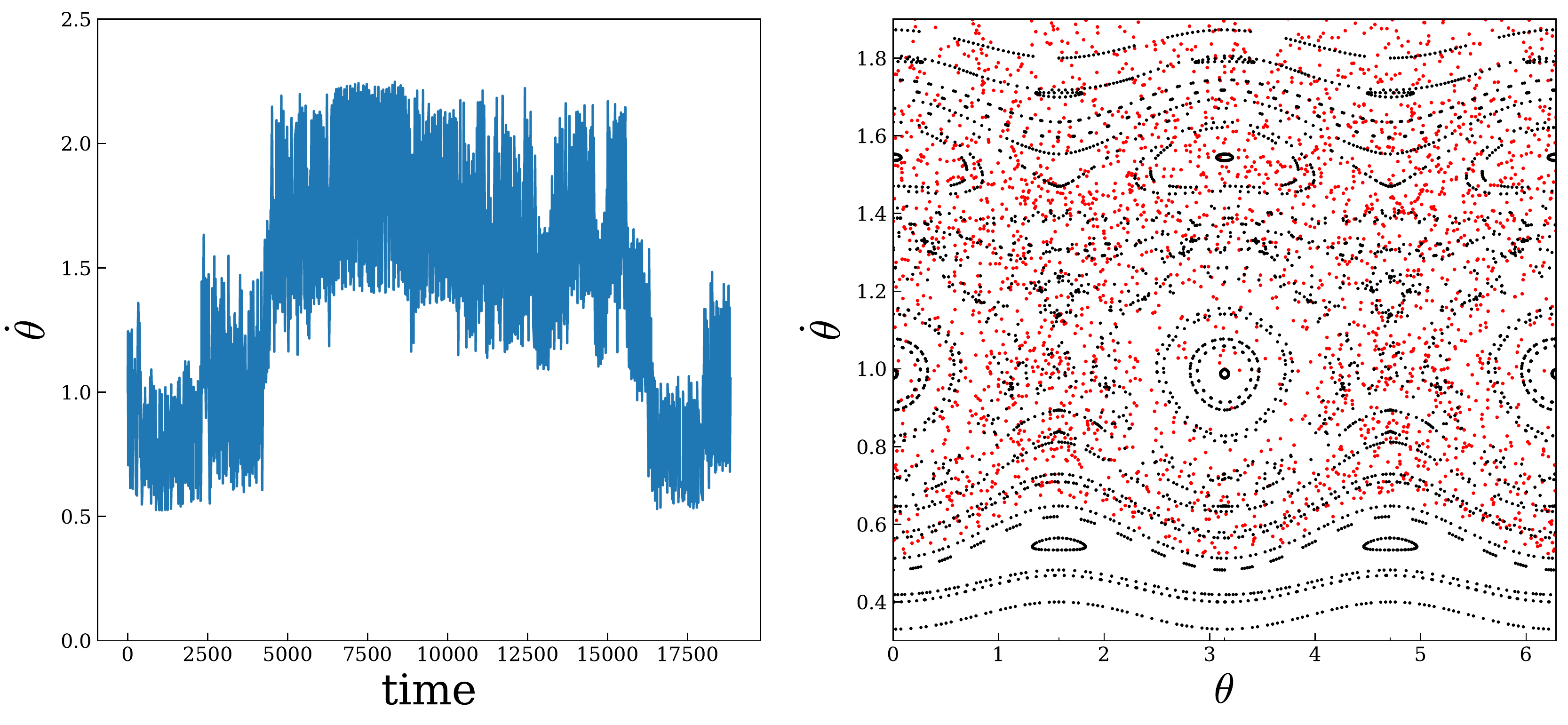}
    \captionsetup{width=\onelinewidth}
    \captionsetup{font=footnotesize}
    \caption{In the left panels we present the time evolution of $\dot\theta$ for the spin-orbit problem \eqref{eq:cHspinorbit} in the non--resonant case $\Omega = \sqrt{2}$ for the samples $\lambda = 0.01$ (top) and $\lambda = 0.2$ (bottom).
    The right panels show the Poincar\'e sections of the full model (red dots) compared to that of the conservative model $\lambda=0$ (black dots).}%
    \label{fig:spinorbit5}
\end{figure}

A direct comparison of Figure~\ref{fig:spinorbit5} and~\cite[Figure 5]{gkolias2017hamiltonian} will show that the Poincar\'e section of the conservative models and the trajectories, albeit qualitatively similar, are  not the same. This is due to the fact that, with the exception of $\lambda$ and $\Omega$, we do not know which value of the parameters is used in~\cite{gkolias2017hamiltonian}. For our comparison we opted for selecting a configuration of parameters producing a qualitatively similar conservative phase space. Remarkably, one can observe the same kind of transition to chaotic regime as the coupling parameter grows.

Finally, Figure~\ref{fig:spinorbit7} depicts the capture of the system into a resonance, in analogy to~\cite[Figure 7]{gkolias2017hamiltonian}. In this case we use the same values as in Figure~\ref{fig:spinorbit} with the exception of $a = \mu = 10^{-3}$, $\lambda = 10^{-4}$. Also in this case we can observe that our technique allows for an accurate description of the qualitative behavior.

In all cases the integration has been performed using the second--order contact Hamiltonian integrator with the rather large time step of $0.314$. However, we remark that no qualitative difference was noticed when using either a smaller or a slightly larger time step, nor when using higher order integrators.

\begin{figure}[ht!]
    \centering
    \includegraphics[width=\onelinewidth]{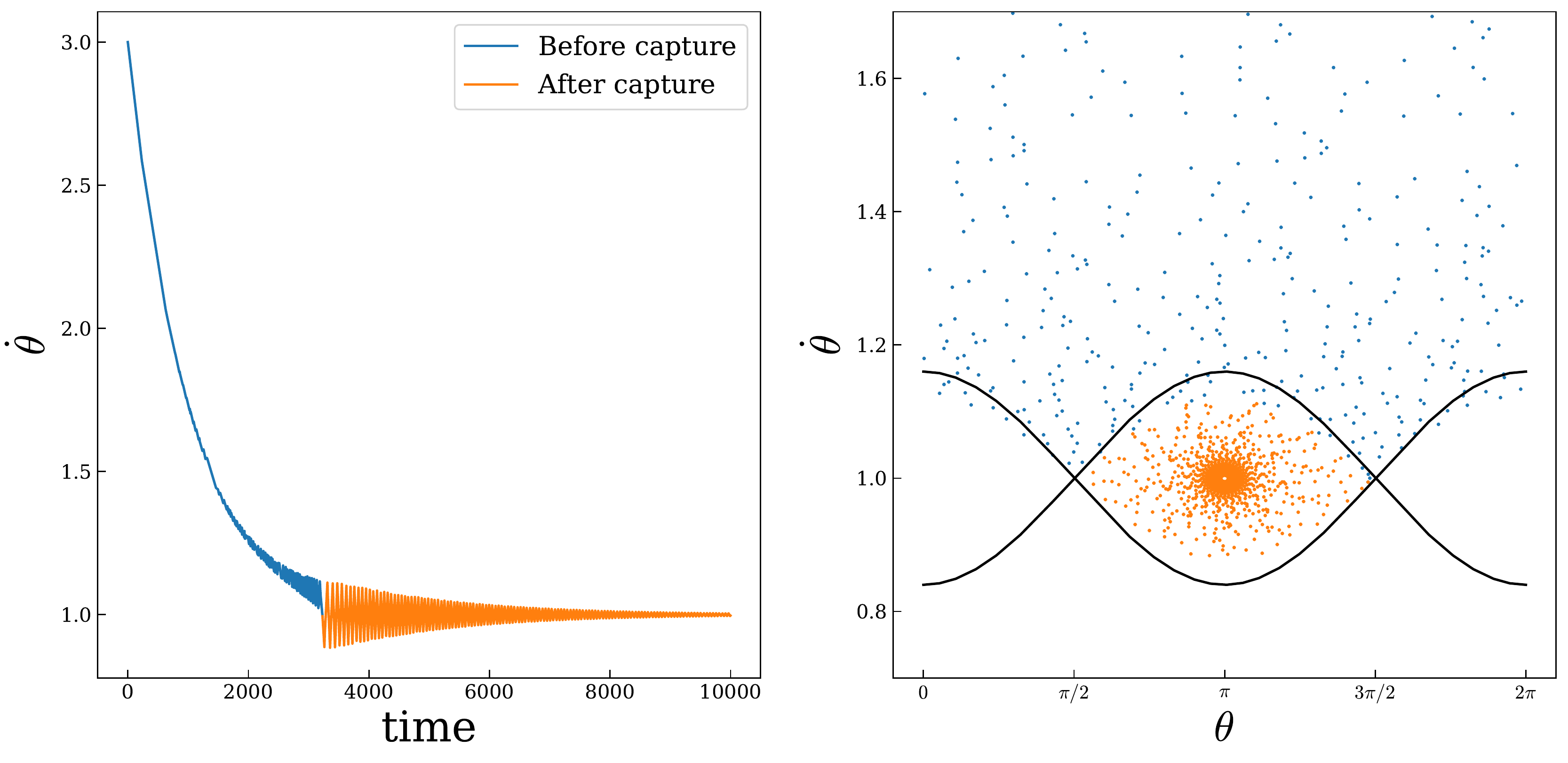}
    \captionsetup{width=\onelinewidth}
    \captionsetup{font=footnotesize}
    \caption{Time evolution of $\dot\theta$ for the spin-orbit problem \eqref{eq:cHspinorbit} in the case of a capture into the $1:1$ resonance (left).
             Poincar\'e section for this scenario, showing the capture into the synchronous resonance (right).}%
    \label{fig:spinorbit7}
\end{figure}

\subsection{The Lane--Emden equation}%
	\label{sec:laneemden}
The Lane--Emden equation
    \begin{equation}
   y''(x)+\frac{2}{x}\,y'(x)+y^n(x)=0, \qquad y(0)=1, \quad y'(0)=0
   \label{eq:laneemden}
    \end{equation}
is a nonlinear singular equation that is widely used in physics
to model isothermal gas spheres, such as e.g.~stars~\cite{chandrasekhar1957introduction}.
In equation~\eqref{eq:laneemden}, $y$ is a dimensionless variable related to the density of the star, and $x$ is a dimensionless distance from the center, and the density is normalised so that the central density is 1.
Finally, the integer $n$, called the barotropic index, depends on the nature of the gas.

Clearly, for $n\neq 0,1$ the Lane--Emden equation is nonlinear. Besides, due to the $1/x$ term, it is singular in the initial condition.
Therefore, the study of solutions of equation~\eqref{eq:laneemden} is at the same time physically important and mathematically challenging.
For this reason a large number of {numerical} schemes have been proposed in the literature,
based on different approaches such as e.g.~series expansions, spectral methods, perturbation techniques, neural network methods, and so on (see~\cite{ahmad2017neural,parand2010lagrangian} for comprehensive lists of the techniques used so far).

For us the relevance of equation~\eqref{eq:laneemden} lies in the observation that such equation belongs to the class~\eqref{eq:geneq} and therefore it has a natural description in terms of contact geometry. In fact, the contact Hamiltonian for the Lane--Emden equation is
    \begin{equation}%
	\label{eq:cHLE}
    \cH=\frac{p^2}{2}+\frac{y^{n+1}}{n+1}+\frac{2}{x}s\,,
    \end{equation}
which is of the type~\eqref{eq:mainCH}, and therefore can in principle be treated using the contact integrators constructed above. 

We will see that, despite the singularity in the last term of~\eqref{eq:cHLE}, which in principle prevents the applicability of Proposition~\ref{errorpropagationprop}, we can still perform a 
heuristic error analysis of the kind discussed in Section~\ref{sec:modifiedhamiltonian}. Using the
exact solutions which are available for the cases $n=0,1,5$~\cite{pandey2012solution}, {we can observe} that the corresponding predictions are still in agreement with the real error.

\subsubsection{Error Analysis}%
	\label{subLE:error}
As we argued in Section~\ref{sec:modifiedhamiltonian}, one of the advantages of our contact integrators is the fact that they allow for a relatively straightforward error analysis based on the modified Hamiltonian. 
Unfortunately the Lane--Emden model falls into the singular setting exposed in Remark~\ref{Rem:singular1}.
Indeed, computing the modified Hamiltonian corresponding to the second order integrator~\eqref{eq:modifiedhamiltonian} for the Lane--Emden model leads to 
    \begin{equation}
    \Delta \mathcal{H} '=
    \frac{1}{24} \left(
        2 y^{2 n}-n p^2 y^{n-1}
        +\frac{2 p y^n}{x}
        +\frac{2 p^2}{x^2}
        -\frac{2 s}{x^3}    
    \right),
    \end{equation}
Since $x=\# steps\cdot \tau$, we see that
indeed the correction term in the modified Hamiltonian is no longer of order $\cO(\tau^2)$, 
but presents additional terms of order $\cO(\tau)$, $\cO(1)$ and, for $s(0) \neq 0$, $\cO(\tau^{-1})$. 
This is a direct consequence of the appearance of a singularity in the independent variable in the Hamiltonian function~\eqref{eq:cHLE}, as anticipated in Remark~\ref{Rem:singular1}.

Nevertheless, we can still get useful information: if we use the modified Hamiltonian to compute the 
local errors in the coordinates, from~\eqref{errqstep1}--\eqref{errsstep1}, we find
    \begin{align}
	\Delta {y}&=\frac{\tau^3}{12} \bigg\lvert \frac{x y^n (y-n p x)+2 p y}{x^2 y} \bigg\rvert,\label{eq:LEerrorq}\\
	\Delta {p}&= \frac{\tau^3}{24}\bigg\lvert 4 n y^{2 n-1} - \frac{n p y^{n-2} ((n-1) p x-2 y)}{x} - \frac{2 p}{x^3} \bigg\rvert,\label{eq:LEerrorp}\\
    \Delta {s}&= \frac{\tau^3}{24}\bigg\lvert n p^2 y^{n-1} + 2 y^{2 n} - \frac{2 \left(p^2 x+s\right)}{x^3} \bigg\rvert.\label{eq:LEerrors}
    \end{align}
\noeqref{eq:LEerrorq}\noeqref{eq:LEerrorp}\noeqref{eq:LEerrors}
Here, since the singularity is at $x=0$, it is crucial to control the error in the first step. For instance, a direct application of the standard fourth order Runge--Kutta method cannot overcome this difficulty.

To quantify the error at the initial step, we recall that the initial condition for the integration of the Lane--Emden equation is $y(0)=1,p(0)=0$, and that at the initial step we have $x=\tau$. Thus~\eqref{eq:LEerrorq}--\eqref{eq:LEerrors} for the initial step give
\begin{align}
	\Delta {y}&= \frac{\tau^2}{12},\label{eq:LEerrorqinitial}\\
	\Delta {p}&=\frac{n \tau ^3}{3},\label{eq:LEerrorpinitial}\\
	\Delta {s}&= \frac{1}{12} \left|s(0) -\tau^3\right|\label{eq:LEerrorsinitial}\,,
\end{align}
which are all bounded and the errors in the physical variables are small for small $\tau$. 

Note that having an explicit estimate for the error at the first step may be used to compensate for such error in more sophisticated ways, for example to choose modified initial conditions to reduce the error.

In summary, we note again that, due to the singularity in time, 
the method is {no longer of} second order in $\tau$
and Proposition~\ref{errorpropagationprop} does not hold.
Therefore one should study the error propagation according to equation~\eqref{serieserror}, which is beyond the scope of the present analysis.
However, the initial errors are bounded, and so are the ensuing errors at later steps, as the numerical experiments below show for the cases $n=0,1,5$. 

\subsubsection{Numerical results}%
	\label{subLE:numerical}

The results of the numerical error analysis for the values of $n=0,1,5$, for which we have exact solutions to use for comparisons, are presented in Figures~\ref{fig:laneemden=0},~\ref{fig:laneemden=1},~\ref{fig:laneemden=5}. 
In all cases we used the second order contact Hamiltonian integrator with a relatively large time step $\tau=0.2$.
We can clearly observe that the numerical trajectories follow the exact solutions closely in all the cases. For small time step the error is well below the upper bound estimated using the modified Hamiltonian. Despite the inapplicability of Proposition~\ref{errorpropagationprop}, the error bound appears close when not above the numerical error even when using a large timestep.

Finally, Figure~\ref{fig:laneemdenbundle} shows the bundle of Lane--Emden densities for $n=0,\dots,9$. In all cases, we imposed the natural initial condition $(p_0, y_0) = (0, 1)$ and were able to integrate the solutions without requiring any special precaution.

\begin{figure}
    \centering
    \includegraphics[width=\onelinewidth]{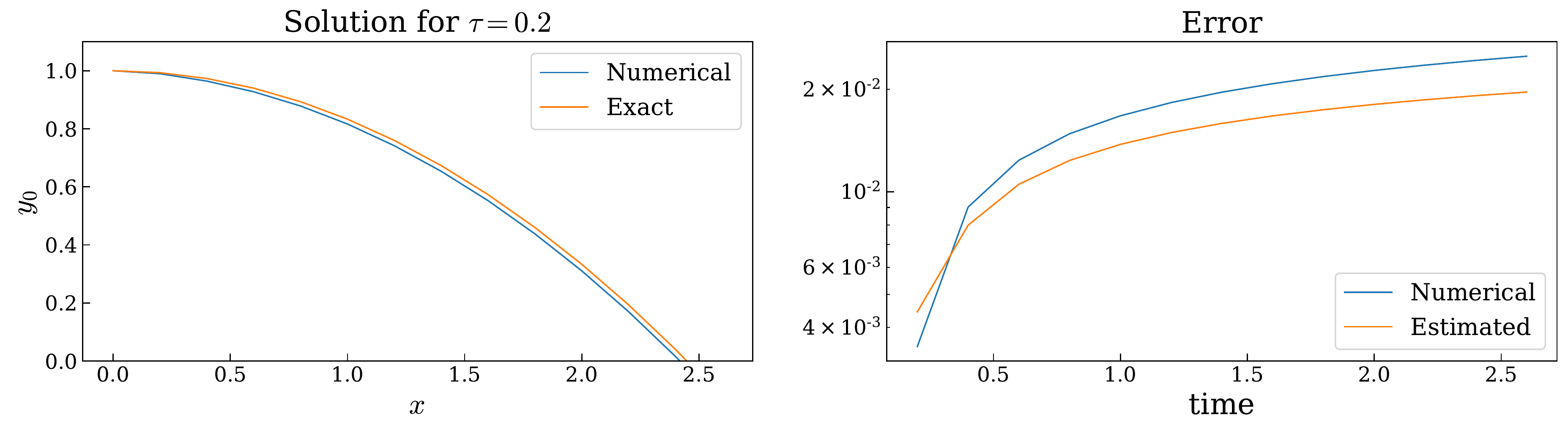}
    \includegraphics[width=\onelinewidth]{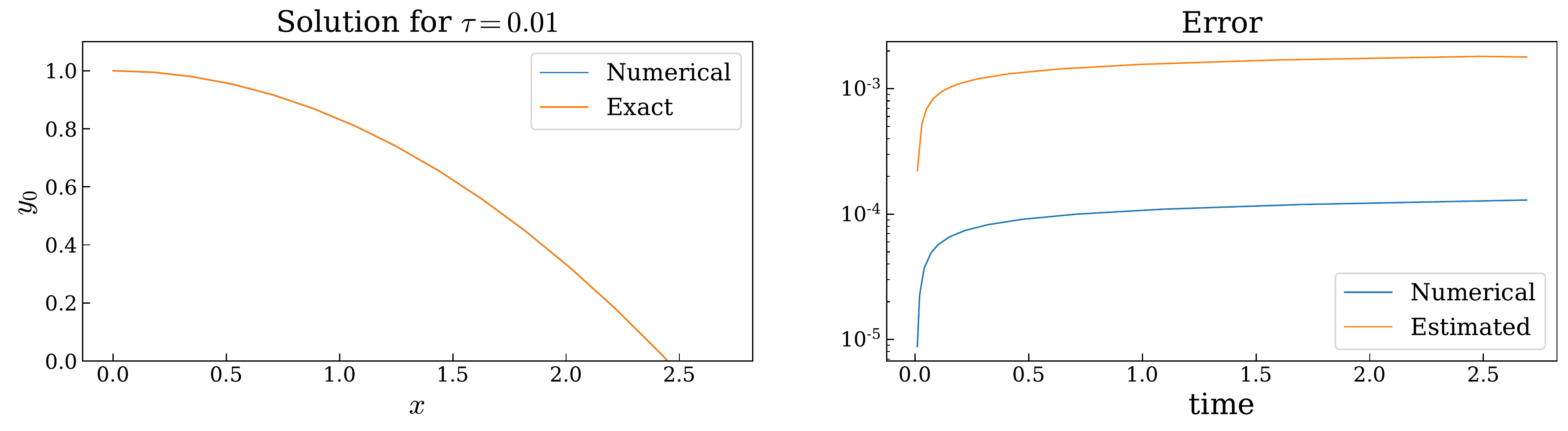}
	\captionsetup{font=footnotesize}
	\captionsetup{width=\onelinewidth}
    \caption{Exact and numerical solution of the Lane--Emden model \eqref{eq:cHLE} with $n=0$ (left). Plot of the total error against the one estimated using the modified Hamiltonian (right).}%
    \label{fig:laneemden=0}
\end{figure}

\begin{figure}
    \centering
    \includegraphics[width=\onelinewidth]{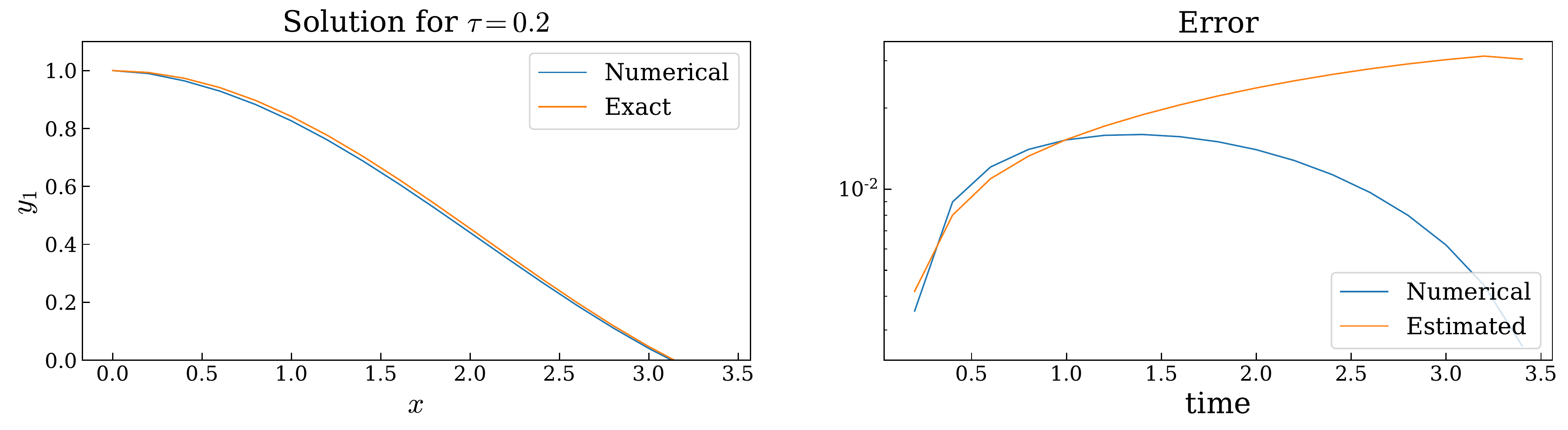}
    \includegraphics[width=\onelinewidth]{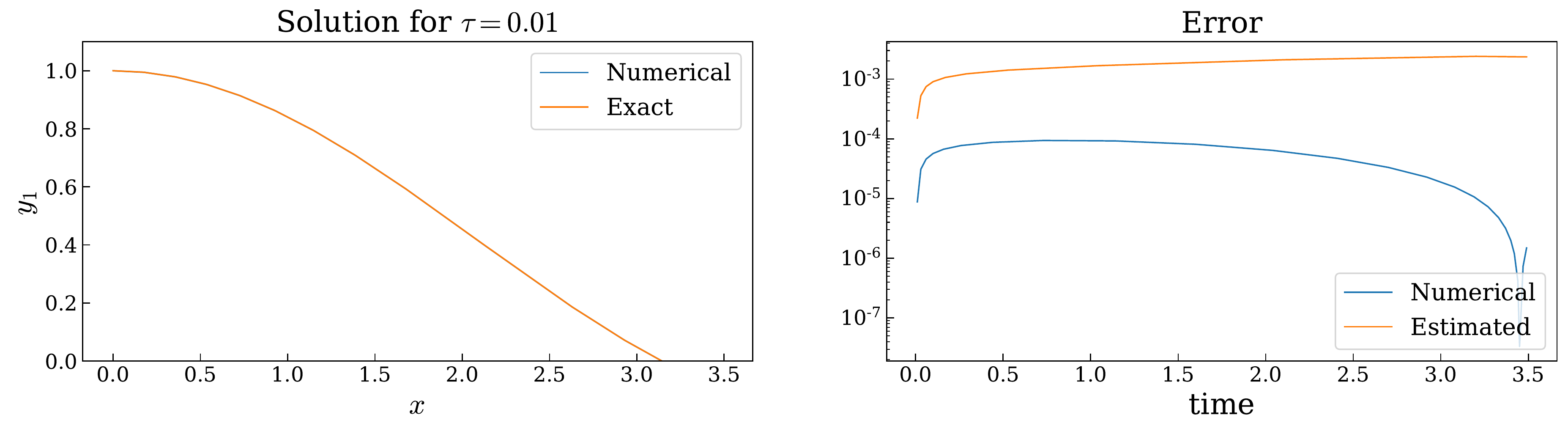}
	\captionsetup{font=footnotesize}
	\captionsetup{width=\onelinewidth}
    \caption{Exact and numerical solution of the Lane--Emden model \eqref{eq:cHLE} with $n=1$ (left). Plot of the total error against the one estimated using the modified Hamiltonian (right).}%
    \label{fig:laneemden=1}
\end{figure}

\begin{figure}
    \centering
    \includegraphics[width=\onelinewidth]{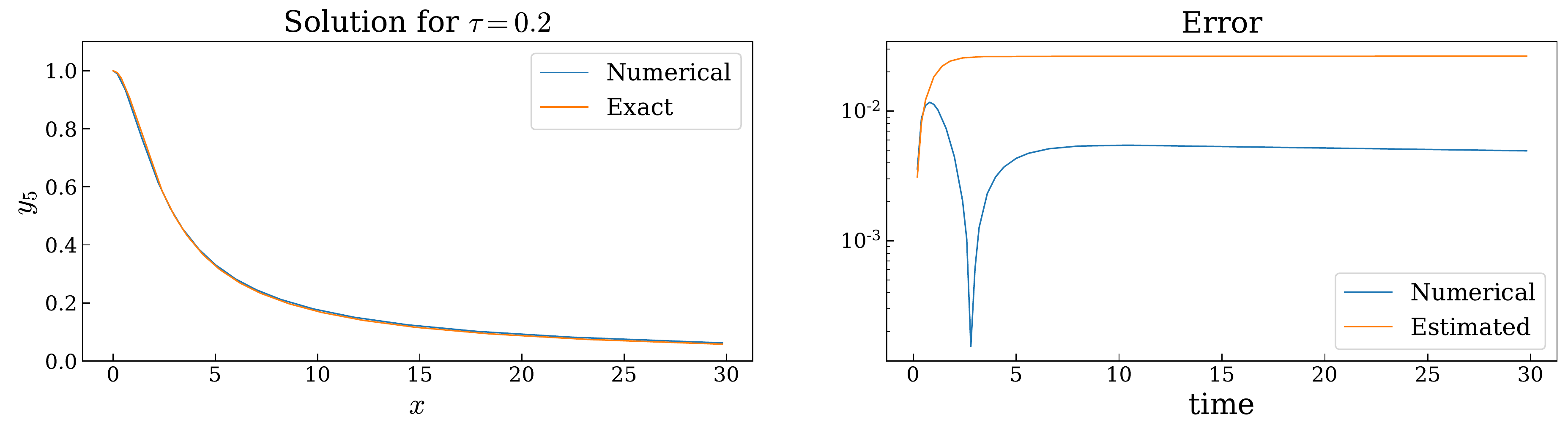}
    \includegraphics[width=\onelinewidth]{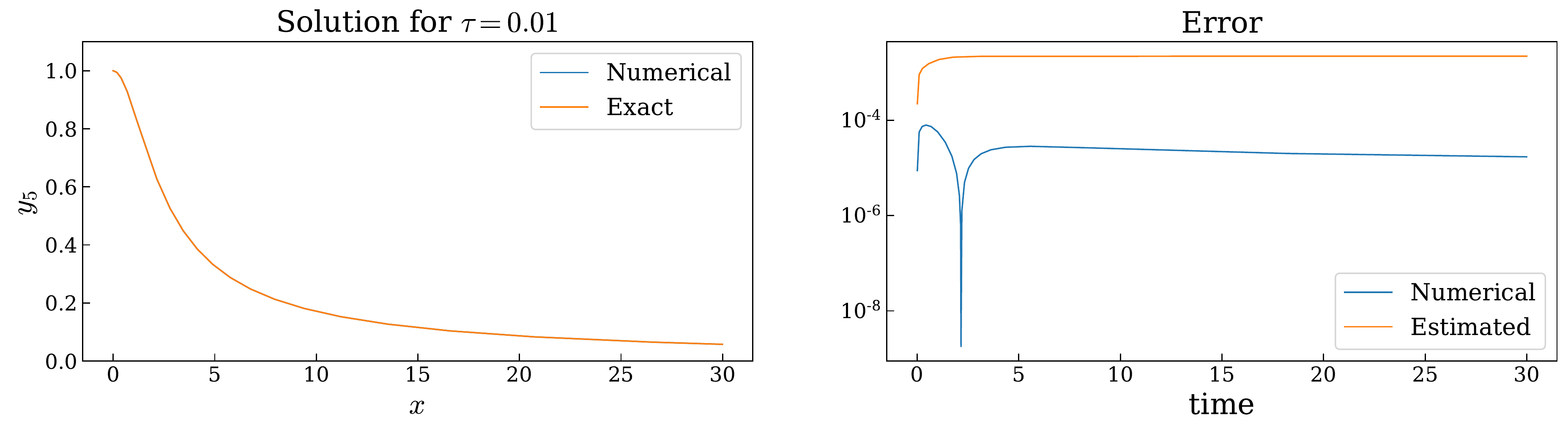}
    \captionsetup{font=footnotesize}
	\captionsetup{width=\onelinewidth}
    \caption{Exact and numerical solution of the Lane--Emden model \eqref{eq:cHLE} with $n=5$ (left). Plot of the total error against the one estimated using the modified Hamiltonian (right).}%
    \label{fig:laneemden=5}
\end{figure}

\begin{figure}
    \centering
    \includegraphics[width=\onelinewidth]{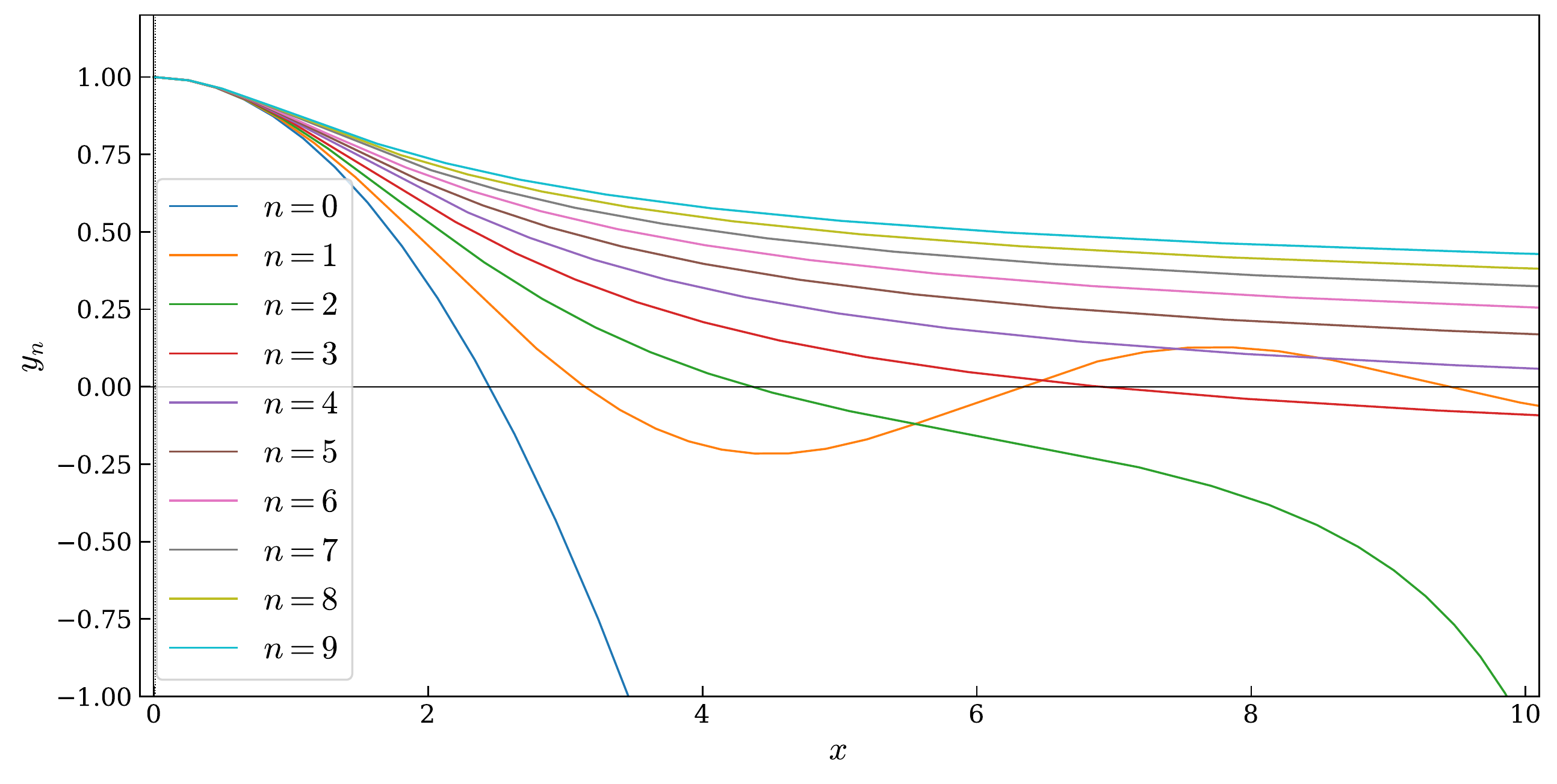}
	\captionsetup{font=footnotesize}
	\captionsetup{width=\onelinewidth}
    \caption{Numerical solutions of the Lane--Emden model \eqref{eq:cHLE} with $n=0,\dots,9$ and $\tau=0.01$ (left).}%
    \label{fig:laneemdenbundle}
\end{figure}

\section{Conclusions}%
	\label{sec:conclusions}
In this work we have considered the class of equations~\eqref{eq:geneq}
and showed that it admits a natural Hamiltonisation in the context of
contact geometry. 

We have argued that such Hamiltonisation is important both for analytical and for numerical studies of the trajectories, and then we have proceeded with a thorough analysis of the numerical case.
In particular, we have developed variational and Hamiltonian integrators for equations of the type~\eqref{eq:geneq} and we have applied them to some important examples arising in celestial mechanics: 
the modified Kepler problem, 
the spin--orbit model and the Lane--Emden equation.
In all the cases studied we have seen that our geometric framework simplifies the numerical analysis, and provides an explicit modified Hamiltonian, which can be used to obtain error estimates and as a starting point for further refinements.

It is important to emphasise at this point some of the different strengths and limitations of the presented methods. 

The variational integrators are very general, but their implementation is more cumbersome and may require specific adaptations to the problems. They can be applied to any regular contact Hamiltonian, including non--separable ones (e.g.~the integration of Schwarzschild geodesics with and without damping or the weak turbulent non--linear Schr\"odinger equations discussed in~\cite{Tao16}). The variational approach generally delivers implicit integrators, which is reflected in that they are significantly more stable in certain cases.

The contact Hamiltonian integrators can only be applied to separable Hamiltonians, but within this class they are extremely versatile and easy to implement. They are explicit methods, making them faster, but sometimes less stable, than their variational counterparts.

The theoretical analysis of the orbits of systems of type~\eqref{eq:geneq} based on the contact Hamiltonian and Lagrangian formulations falls in a very active field of research; we postpone such analysis for the models presented here to future works.

\subsection*{Acknowledgements}
The authors would like to thank the Bernoulli Institute for the hospitality. This research was partially supported by MS's starter grant and NWO Visitor Travel Grant 040.11.698 that sponsored the visit of AB at the Bernoulli Institute. MS and FZ research is supported by the NWO project 613.009.10. MV is supported by the DFG through the SFB Transregio 109 ``Discretization in Geometry and Dynamics''.

\section*{Compliance with ethical standards}
\subsection*{Conflict of interest} The authors declare that they have no conflict of interest.

\bibliographystyle{abbrvnat_mv}
\bibliography{contact_int.bib}

\end{document}